\documentclass[11pt]{article}

\usepackage{amsfonts,amssymb,amsmath,amsthm,epsfig,euscript,epstopdf}
\usepackage[noadjust]{cite}
\usepackage{filecontents}

\newtheorem{theorem}{Theorem}
\newtheorem{lemma}[theorem]{Lemma}
\newtheorem{corollary}[theorem]{Corollary}

\long\def\symbolfootnote[#1]#2{\begingroup
\def\thefootnote{\fnsymbol{footnote}}\footnote[#1]{#2}\endgroup}

\newcommand{\la}{\lambda}
\newcommand{\La}{\Lambda}

\newcommand{\des}{\mathrm{des}}

\newcommand{\LRmin}[1]{\mathrm{LRmin}(#1)}

\newcommand{\nth}[1][n]{{#1}^{\mathrm{th}}}

\newcommand{\sg}{\sigma}

\newcommand{\cref}[1]{Corollary \ref{corollary:#1}}

\newcommand{\Floor}[1][n/2]{\left \lfloor #1 \right \rfloor}

\newcommand{\red}{\mathrm{red}}

\newcommand{\sgn}[1]{\mathrm{sgn}(#1)}

\newcommand{\Gmch}{\text{$\Gamma$-$\mathrm{mch}$}}

\newcommand{\fig}[2]{\begin{figure}[ht]
\centerline{\scalebox{.66}{\epsfig{file=#1.eps}}}
\caption{#2}
\label{fig:#1}
\end{figure}}

\title{Generating functions for descents over permutations 
which avoid sets of consecutive patterns.}

\author{
Quang T. Bach \\
\small Department of Mathematics\\[-0.8ex]
\small University of California, San Diego\\[-0.8ex]
\small La Jolla, CA 92093-0112. USA\\[-0.8ex]
\small \texttt{qtbach@ucsd.edu}
\and
\and
Jeffrey B. Remmel \\
\small Department of Mathematics\\[-0.8ex]
\small University of California, San Diego\\[-0.8ex]
\small La Jolla, CA 92093-0112. USA\\[-0.8ex]
\small \texttt{remmel@math.ucsd.edu}
\and
}

\date{\small Submitted: Date 1;  Accepted: Date 2;
 Published: Date 3.\\
\small MR Subject Classifications: 05A15, 05E05 \\
keywords: pattern avoidance, consecutive pattern, permutation, pattern match, descent, left to right minimum, symmetric polynomial, exponential generating function}

\begin{document}

\maketitle

\begin{abstract}
We extend the reciprocity method of Jones and Remmel \cite{JR2,JR3}
to study generating functions of the form   
$\sum_{n \geq 0} \frac{t^n}{n!} 
\sum_{\sg \in \mathcal{NM}_n(\Gamma)}x^{\LRmin{\sg}}y^{1+\des(\sg)}$ where 
$\Gamma$ is a set of permutations which start with 1 and have 
at most one descent, 
$\mathcal{NM}_n(\Gamma)$  
is the set of permutations $\sg$ in the symmetric group $\mathfrak{S}_n$ 
which have no $\Gamma$-matches, $\des(\sg)$ is the number 
of descents of $\sg$ and $\LRmin{\sg}$ is the number of left-to-right 
minima of $\sg$. We show that this generating function  
is of the form $\left( \frac{1}{U_{\Gamma}(t,y)}\right)^x$ where 
$U_{\Gamma}(t,y) = \sum_{n\geq 0}U_{\Gamma,n}(y) \frac{t^n}{n!}$ 
and the coefficients 
$U_{\Gamma,n}(y)$ satisfy some simple recursions in 
the case where $\Gamma$ equals $\{1324,123\}$, 
$\{1324 \cdots p,12 \cdots (p-1)\}$ and $p \geq 5$, or $\Gamma$ is the 
set of permutations $\sg = \sg_1 \cdots \sg_n$ 
of length $n=k_1+k_2$ where $k_1,k_2 \geq 2$, $\sg_1 =1$, $\sg_{k_1+1}=2$, and 
$\des(\sg) =1$. 
\end{abstract}

\section{Introduction}

Let $\mathfrak{S}_n$ denote the symmetric group of all permutations 
of $\{1, \ldots, n\}$.  If 
$\sg = \sg_1 \cdots \sg_n \in \mathfrak{S}_n$, we say that $i$ is a descent of 
$\sg$ if $\sg_i > \sg_{i+1}$ and $\sg_j$ is a left-to-right minimum of 
$\sg$ if  $\sg_j < \sg_i$ for all $i < j$.  We let 
$\des(\sg)$ be the number of descents of $\sg$ and $\LRmin{\sg}$ be the number of left-to-right minima of $\sg$. Given a sequence $\alpha= \alpha_1 \cdots \alpha_n$ of distinct integers, the \emph{reduction} of $\alpha$, $\red(\alpha)$, 
is the permutation in $\mathfrak{S}_n$ found by replacing the
\begin{math}i^{\textrm{th}}\end{math} smallest integer that appears in \begin{math}\alpha\end{math} by \begin{math}i\end{math}.   For
example, if \begin{math}\alpha = 9~2~7~4~5\end{math}, then 
$\red(\alpha) = 51423$. Let $\Gamma$ be a set of permutations. We say 
that a permutation $\sg = \sg_1 \cdots \sg_n \in \mathfrak{S}_n$ 
has a $\Gamma$-match starting at 
position $i$ if there is a $j \geq 1$ such that 
$\red(\sg_i \sg_{i+1} \cdots \sg_{i+j}) \in \Gamma$. We let 
$\Gmch(\sg)$ denote the number of $\Gamma$-matches in $\sg$. 
We let $\mathcal{NM}_n(\Gamma)$  
be the set of permutations $\sg$ in the symmetric group $\mathfrak{S}_n$ such 
that $\Gmch(\sg) =0$.

The main goal of this paper is 
to study the generating function 
$$\mbox{NM}_{\Gamma}(t,x,y)=\sum_{n \geq 0} \frac{t^n}{n!} 
\sum_{\sg \in \mathcal{NM}_n(\Gamma)}x^{\LRmin{\sg}}y^{1+\des(\sg)}$$ 
in the case where $\Gamma$ is a set of permutations such that 
for each $\alpha \in \Gamma$, $\alpha$ starts with 1 and 
$des(\alpha) \leq 1$.  In the special case where 
$\Gamma$ consists of a single permutation 
$\tau$, we will denote $\mbox{NM}_{\Gamma}(t,x,y)$ simply as 
$\mbox{NM}{\tau}(t,x,y)$. Jones and Remmel \cite{JR1}
showed that if every permutation in $\Gamma$ starts with 1, then  
we can write $\mbox{NM}_{\Gamma}(t,x,y)$ in the form 
$\left( \frac{1}{U_{\Gamma}(t,y)}\right)^x$ where 
$$U_{\Gamma}(t,y) = \sum_{n\geq 0}U_{\Gamma,n}(y) \frac{t^n}{n!}.$$

There is a considerable literature on the 
generating function $\mbox{NM}_{\Gamma}(t,1,1)$ of 
permutations that consecutively avoid a pattern 
or set of patterns. See for example,\cite{AAM,B1,B2,DK,DR,EN,EN2,EKP,GJ,Kit1,Kitbook,MenRem}. For the most part, these papers do not  
consider generating functions of the form 
$\mbox{NM}{\tau}(t,1,y)$ or $\mbox{NM}{\tau}(t,x,y)$. An exception 
is the work on enumeration schemes of Baxter \cite{B1,B2}, who gave 
general methods to enumerate patterns avoiding vincular patterns according 
to various permutations statistics. Our approach is 
to use the reciprocity method of Jones and Remmel.

Jones and Remmel \cite{JR,JR2,JR3} 
developed what they called the reciprocity method 
to compute the generating function 
$\mbox{NM}{\tau}(t,x,y)$ for certain families of 
permutations $\tau$ such that $\tau$ starts with 
1 and $\des(\tau) =1$. The basic idea of their approach 
is as follows. First one writes 
\begin{equation}\label{eq:I3}
U_{\tau}(t,y) = 
\frac{1}{1+\sum_{n \geq 1} \mbox{NM}_{\tau,n}(1,y) \frac{t^n}{n!}}.
\end{equation} One can then use the homomorphism method 
to give a combinatorial   
interpretation of the right-hand side of (\ref{eq:I3}) which can 
be used to find simple recursions for  
the coefficients $U_{\tau,n}(y)$. 
The homomorphism method derives generating functions for 
various permutation statistics by 
applying a ring homomorphism defined on the 
ring of symmetric functions \begin{math}\Lambda\end{math}  
in infinitely many variables \begin{math}x_1,x_2, \ldots \end{math} 
to simple symmetric function identities such as 
\begin{equation*}
H(t) = 1/E(-t),
\end{equation*}
where $H(t)$ and $E(t)$ are the generating functions for the homogeneous and elementary 
symmetric functions, respectively:
\begin{equation}\label{genfns}
H(t) = \sum_{n\geq 0} h_n t^n = \prod_{i\geq 1} \frac{1}{1-x_it},~~~~ E(t) = \sum_{n\geq 0} e_n t^n = \prod_{i\geq 1} 1+x_it.
\end{equation}
In their case, Jones and Remmel
 defined a homomorphism \begin{math}\theta\end{math} on 
\begin{math}\Lambda\end{math} by setting 
\begin{displaymath}\theta(e_n) = \frac{(-1)^n}{n!} \mbox{NM}_{\tau,n}(1,y).\end{displaymath}
Then 
\begin{displaymath}\theta(E(-t)) = {\sum_{n\geq 0} \mbox{NM}_{\tau,n}(1,y) \frac{t^n}{n!}} = \frac{1}{U_\tau(t,y)}.\end{displaymath}
Hence 
$$U_\tau(t,y) = \frac{1}{\theta(E(-t))} = \theta(H(t)),$$
which implies that 
\begin{equation*}
n!\theta(h_n) = U_{\tau,n}(y).
\end{equation*}
Thus if we can compute $n!\theta(h_n)$ for all $n \geq 1$, then we can 
compute the polynomials $U_{\tau,n}(y)$ and the generating function 
$U_{\tau}(t,y)$, which in turn allows us to compute 
the generating function $\mbox{NM}_{\tau}(t,x,y)$. 
Jones and Remmel \cite{JR2,JR3} showed that one can interpret 
$n!\theta(h_n)$ as a certain signed sum of weights of filled labeled brick tabloids when $\tau$ starts with 1 
and $\des(\tau)=1$. Then they showed how such a  
combinatorial interpretation allowed them to prove  
that for certain families of such permutations $\tau$, 
the $U_{\tau,n}(y)$'s satisfied certain simple recursions.

The main purpose of this paper is to extend 
the methods of Jones and Remmel \cite{JR2,JR3} so 
that one can compute $U_{\Gamma,n}(y)$.  In 
our case we assume that if $\tau \in \Gamma$, then 
$\tau$ starts with 1 and $\des(\tau) \leq 1$.  One 
of the most interesting cases from our point of 
view is the case when $\Gamma$ contains an identity  permutation 
$12 \cdots (k+1)$ where $k \geq 2$. In such a case, 
the underlying set of weighted 
filled labeled brick tabloids which we use to interpret 
$U_{\Gamma,n}(y)$ has the property that all 
the bricks have size less than or equal to $k$. This results 
in a significant difference between the recursions 
satisfied by $U_{\tau,n}(y)$ and the 
recursions satisfied by $U_{\{\tau,12 \cdots (k+1)\},n}(y)$.

For example, in \cite{JR2}, Jones and Remmel  studied  the generating 
functions $\mbox{NM}_{\tau}(t,x,y)$ for permutations $\tau$ of the form 
$\tau = 1324\cdots p$ where $p \geq 4$. That is, $\tau$ arises 
from the identity permutation
by transposing 2 and 3. Using the reciprocity method, 
they proved that $U_{1324,1}(y)=-y$ and 
for $n \geq 2$, 
\begin{equation*}
U_{1324,n}(y) = (1-y)U_{1324,n-1}(y)+ \sum_{k=2}^{\lfloor n/2 \rfloor} 
(-y)^{k-1} C_{k-1} U_{1324,n-2k+1}(y)
\end{equation*}
where \begin{math}C_k = \frac{1}{k+1}\binom{2k}{k}\end{math} is the \begin{math}k^{th}\end{math} 
Catalan number.  They also proved that for  any $p \geq 5$, 
$U_{1324 \cdots p,n}(y) =-y$ and for $n \geq 2$, 
\begin{equation*}
U_{1324\cdots p,n}(y)=(1-y)U_{1324\cdots p,n-1}(y)+\sum_{k=2}^{\lfloor\frac{n-2}{p-2}\rfloor+1}(-y)^{k-1}U_{1324\cdots p,n-((k-1)(p-2)+1)}(y).
\end{equation*}

We will 
prove the following two theorems.  
\begin {theorem} \label{thm:1324123}
Let $\Gamma = \{1324,123\}$. Then 
$$\mbox{NM}_\Gamma(t,x,y)=\left(\frac1{U_\Gamma(t,y)}\right)^x \text{ where }U_\Gamma(t,y)=1+\sum_{n\geq1}U_{\Gamma,n}(y)\frac{t^n}{n!},$$
$U_{\Gamma,1}(y)=-y$, and for $n \geq 2,$
\begin{align*}
\displaystyle  U_{\Gamma,n}(y) &= -yU_{\Gamma,n-1}(y) -yU_{\Gamma,n-2}(y)
+ \sum_{k=2}^{\Floor }(-y)^{k}C_{k-1}U_{\Gamma, n-2k}(y).
\end{align*} 
\end{theorem}

\begin {theorem} \label{thm:1324p}
Let $\Gamma = \{1324\dots p,123\dots p-1\}$ where $p \geq 5$.  Then 
$$\mbox{NM}_\Gamma(t,x,y)=\left(\frac1{U_\Gamma(t,y)}\right)^x \text{ where }U_\Gamma(t,y)=1+\sum_{n\geq1}U_{\Gamma,n}(y)\frac{t^n}{n!},$$
$U_{\Gamma,1}(y)=-y$, and for $n \geq 2,$
\begin{align*}
\displaystyle  U_{\Gamma,n}(y) &= \sum_{k=1}^{p-2}(-y)U_{\Gamma,n-k}(y) + \sum_{k=1}^{p-2}\sum_{m=2}^{\lfloor\frac{n-k}{p-2}\rfloor }(-y)^{m}U_{\Gamma, n-k-(m-1)(p-2)}(y). 
\end{align*}
\end{theorem}

Note that both Theorems \ref{thm:1324123} and  \ref{thm:1324p} show 
that the reciprocity method applies even in cases where 
$\Gamma$ is a family that contains permutations of 
different lengths.  In the case of Theorem 1, the polynomials 
$U_{\{1324,123\},n}(-y)$ are the polynomials in 
the sequences  A039598 and A039599 in On-line Encyclopedia of 
Integer Sequences \cite{oeis} up to a power of $y$.  The 
polynomials in sequences A039598 and A039599 are related 
to the expansions of the powers of $x$ in terms of 
the Chebyshev polynomials of the second kind. 
We will give a bijection between our combinatorial 
interpretation of $U_{\{1324,123\},2n}(-y)$ and one of 
the known combinatorial interpretations for A039599, and a bijection 
between our combinatorial 
interpretation of $U_{\{1324,123\},2n+1}(-y)$ and one of 
the known combinatorial interpretations for A039598. This will allow us 
to give closed expressions for the polynomials $U_{\{1324,123\},n}(y)$. 
That is, we will prove that for 
all $n \geq 0$, 
\begin{eqnarray*}
U_{\{1324,123\},2n}(y) &=& 
\sum_{k=0}^n \frac{(2k+1)\binom{2n}{n-k}}{n+k+1}(-y)^{n+k+1} \ \mbox{and} \\
U_{\{1324,123\},2n+1}(y) &=& 
\sum_{k=0}^n \frac{2(k+1)\binom{2n+1}{n-k}}{n+k+2}(-y)^{n+k}.
\end{eqnarray*}

Another 
example is the following.  
Let $k_1,k_2 \geq 2$ and $p = k_1+k_2$. 
We consider the family of permutations $\Gamma_{k_1,k_2}$ in $\mathfrak{S}_p$ defined as 
$$\Gamma_{k_1,k_2} = \{\sigma \in \mathfrak{S}_p: \sigma_1=1, \sigma_{k_1+1}=2, \sigma_1 < \sigma_2< \cdots<\sigma_{k_1}~ \&~\sigma_{k_1+1} < \sigma_{k_1+2}< \cdots<\sigma_{p} \}.$$
That is, $\Gamma_{k_1,k_2}$ consists of all permutations $\sg$ of 
length $p$ where 1 is in position 1, 2 is in position $k_1+1$, 
and $\sg$ consists of two increasing sequences, one starting at 
1 and the other starting at 2. Then we shall prove the following 
theorem.  

\begin {theorem} \label{thm:1-2-gen}
Let  $\Gamma = \Gamma_{k_1,k_2}$ where $k_1,k_2 \geq 2$,
$m = \min\{k_1, k_2\}$, and $M = \max\{k_1,k_2\}$.  
Then 
$$\mbox{NM}_{\Gamma}(t,x,y)=\left(\frac1{U_\Gamma(t,y)}\right)^x \text{ where }U_\Gamma(t,y)=1+\sum_{n\geq1}U_{\Gamma,n}(y)\frac{t^n}{n!},$$
$U_{\Gamma,1}(y)=-y$, and for $n \geq 2,$
\begin{align*}
\displaystyle   U_{\Gamma,n}(y) &= (1-y)U_{\Gamma,n-1}(y) -y\binom{n-2}{k_1-1}\left( U_{\Gamma,n-M}(y) +y\sum_{i=1}^{m-1}U_{\Gamma,n-M-i}(y) \right). 
\end{align*}
\end{theorem}

When $k_1 = k_2 = 2,$ Theorem \ref{thm:1-2-gen} gives us the 
following corollary. 

\begin{corollary} \label{cor:13241423}
For $\Gamma = \{1324, 1423\},$ then 
$$\mbox{NM}_\Gamma(t,x,y)=\left(\frac1{U_\Gamma(t,y)}\right)^x \text{ where }U_\Gamma(t,y)=1+\sum_{n\geq1}U_{\Gamma,n}(y)\frac{t^n}{n!},$$
$U_{\Gamma,1}(y)=-y$, and for $n \geq 2,$
$$\displaystyle  U_{\Gamma,n}(y)= (1-y)U_{\Gamma,n-1}(y) - y(n-2)\left( U_{\Gamma,n-2}(y) + yU_{\Gamma,n-3}(y)   \right). $$ 
\end{corollary}

Finally, we shall consider families 
of the form $\Gamma_{k_1,k_2,s} = \Gamma_{k_1,k_2} \cup \{1 \cdots s(s+1)\}$ for some $s \geq \max(k_1,k_2)$.  For example, we will show that 
$$\mbox{NM}_{\Gamma_{2,2,s}}(t,x,y)= \frac{1}{1+ \sum_{n \geq 1}
U_{\Gamma_{2,2,s},n}(y) 
\frac{t^n}{n!}}$$ 
where $U_{\Gamma_{2,2,s},1}(y)=-y$, and for $n \geq 2,$
\begin{align*}U_{\Gamma_{2,2,s},n}(y)&= -yU_{\Gamma_{2,2,s},n-1}(y) - \\
&\ \ \ \sum_{k=0}^{s-2} \left((n-k-1) yU_{\Gamma_{2,2,s},n-k-2}(y)+(n-k-2)
y^2 U_{\Gamma_{2,2,s},n-k-3}(y)\right).
\end{align*}

On the surface, it seems that these recursions are more complicated than 
the recursions for the $U_{\{1324,1423\},n}(y)$'s, but it turns 
out that the resulting polynomials are considerably simpler to analyze. 
For example, we shall give  explicit 
formulas for $U_{\Gamma_{2,2,2},n}(y)$ for all $n \geq 1$. That is, 
we will show that 
\begin{eqnarray*}
U_{\Gamma_{2,2,2},2n}(y) &=& \sum_{i=0}^n (2n-1)\downarrow \downarrow_{n-i} (-y)^{n+i} 
\ \mbox{and} \\
U_{\Gamma_{2,2,2},2n+1}(y) &=& \sum_{i=0}^n (2n) \downarrow \downarrow_{n-i} (-y)^{n+1+i}
\end{eqnarray*}
where for any $x$, $(x)\downarrow \downarrow_{0} =1$ and 
 $(x)\downarrow \downarrow_{k} =x(x-2)(x-4) \cdots (x-2k -2)$ for $k \geq 1$.

The outline of this paper is as follows. 
In Section 2, we will show how to extend the reciprocity method 
of Jones and Remmel \cite{JR2,JR3} to give combinatorial 
interpretations to the polynomials $U_{\Gamma,n}(y)$ in 
the case where all the permutations in $\Gamma$ start with 1 
and have at most one descent. In Section 3, 
we will prove Theorem \ref{thm:1-2-gen} and 
show how to modify it when we add the identity permutation 
in $\mathfrak{S}_{k+1}$ to the corresponding families in the case where 
$k_1 =k_2$. In Section 4, we will prove Theorems \ref{thm:1324123}
and \ref{thm:1324p} and give bijections that will prove 
our closed expressions for the $U_{\{1324,123\},n}(y)$'s.
Finally, in 
Section 5, we will state some open problems and areas for further 
research.

\section{Symmetric Functions}

In this section, we give the necessary background on 
symmetric functions that will be used in our proofs.

A partition of $n$ is a sequence of positive integers 
\begin{math}\la = (\la_1, \ldots ,\la_s)\end{math} such that 
\begin{math}0 < \la_1 \leq \cdots \leq \la_s\end{math} and $n=\la_1+ \cdots +\la_s$. We shall write $\lambda \vdash n$ to denote that $\lambda$ is 
partition of $n$ and we let $\ell(\lambda)$ denote the number of parts 
of $\lambda$. When a partition of $n$ involves repeated parts, 
we shall often use exponents in the partition notation to indicate 
these repeated parts. For example, we will write 
$(1^2,4^5)$ for the partition $(1,1,4,4,4,4,4)$.

Let \begin{math}\Lambda\end{math} denote the ring of symmetric functions in infinitely  
many variables \begin{math}x_1,x_2, \ldots \end{math}. The \begin{math}\nth\end{math} elementary symmetric function \begin{math}e_n = e_n(x_1,x_2, \ldots )\end{math}  and \begin{math}\nth\end{math} homogeneous 
 symmetric function \begin{math}h_n = h_n(x_1,x_2, \ldots )\end{math} are defined by the generating functions given in (\ref{genfns}). 
For any partition \begin{math}\la = (\la_1,\dots,\la_\ell)\end{math}, let \begin{math}e_\la = e_{\la_1} 
\cdots e_{\la_\ell}\end{math} and \begin{math}h_\la = h_{\la_1} 
\cdots h_{\la_\ell}\end{math}.  It is well known that \begin{math}e_0,e_1, \ldots \end{math} is 
an algebraically independent set of generators for \begin{math}\La\end{math}, and hence, 
a ring homomorphism \begin{math}\theta\end{math} on \begin{math}\Lambda\end{math} can be defined  
by simply specifying \begin{math}\theta(e_n)\end{math} for all \begin{math}n\end{math}.

If $\lambda =(\lambda_1, \ldots, \lambda_k)$ is a partition of $n$, 
then a $\lambda$-brick tabloid of shape $(n)$ is a filling 
of a rectangle consisting of $n$ cells with bricks of sizes 
$\lambda_1, \ldots, \lambda_k$ in such a way that no 
two bricks overlap. For example, Figure 
\ref{fig:DIMfig1} shows the six $(1^2,2^2)$-brick tabloids of 
shape $(6)$.

\fig{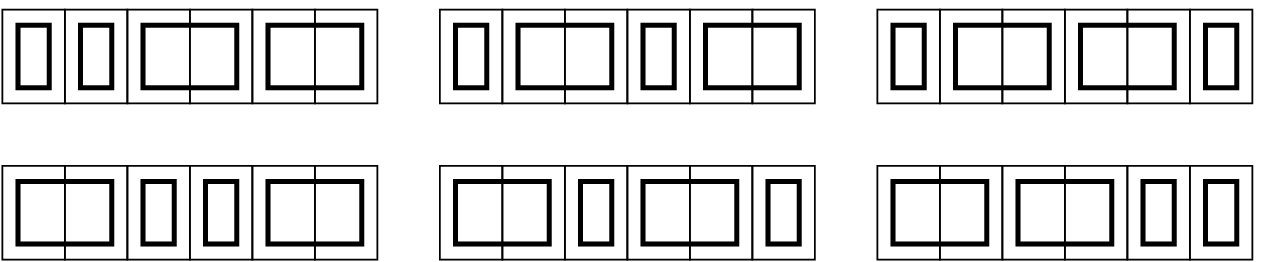}{The six $(1^2,2^2)$-brick tabloids of shape $(6)$.}

Let \begin{math}\mathcal{B}_{\la,n}\end{math} denote the set of \begin{math}\la\end{math}-brick tabloids 
of shape \begin{math}(n)\end{math} and let \begin{math}B_{\la,n}\end{math} be the number of \begin{math}\la\end{math}-brick 
tabloids of shape \begin{math}(n)\end{math}.  If \begin{math}B \in \mathcal{B}_{\la,n}\end{math}, we 
will write \begin{math}B =(b_1, \ldots, b_{\ell(\la)})\end{math} if the lengths of 
the bricks in \begin{math}B\end{math}, reading from left to right, are 
\begin{math}b_1, \ldots, b_{\ell(\la)}\end{math}. For example, the brick 
tabloid in the top right position in Figure \ref{fig:DIMfig1} is 
denoted as $(1,2,2,1)$. 
E\u{g}ecio\u{g}lu and the second author  \cite{Eg1} proved that 
\begin{equation}\label{htoe}
h_n = \sum_{\la \vdash n} (-1)^{n - \ell(\la)} B_{\la,n}~ e_\la.
\end{equation}
This interpretation of $h_n$ in terms of $e_n$ will aid us in describing the coefficients of $\Theta_\Gamma(H(t))=U_\Gamma(t,y)$ described 
in the next section, 
which will in turn allow us to compute the coefficients 
$\mbox{NM}_{\Gamma,n}(x,y)$.

\section{Extending the reciprocity method}

In this section, we will show that one can easily extend the reciprocity 
method of \cite{JR,JR2,JR3} to find a combinatorial 
interpretation for $U_{\Gamma,n}(y)$ in the 
case where $\Gamma$ is a set of permutations which all start with 1 
and have at most one descent. We can assume 
that $\Gamma$ contains at most one permutation 
$\sg$ which is an identity permutation. 
That is, if $12 \cdots s$ and $12 \cdots t$ are in $\Gamma$ 
for some $s < t$, then if we consecutively avoid 
$12 \cdots s$, we automatically consecutively avoid $12 \cdots t$. 
Thus $\mathcal{NM}_n(\Gamma) = \mathcal{NM}_n(\Gamma-\{12 \cdots t\})$ 
for all $n$.

We want give a combinatorial interpretation to 
\begin{equation*}
U_{\Gamma}(t,y)=\frac{1}{\mbox{NM}_{\Gamma}(t,1,y)}  = \frac{1}{1+ \sum_{n \geq 1} 
\frac{t^n}{n!} \mbox{NM}_{\Gamma,n}(1,y)},
\end{equation*}
where 
$$\mbox{NM}_{\Gamma,n}(1,y) = \sum_{\sg \in \mathcal{NM}_n(\Gamma)} y^{1+\des(\sg)}.$$ 

We define a homomorphism $\Theta_{\Gamma}$ on 
 the ring of symmetric functions $\Lambda$ by setting $\Theta_{\Gamma}(e_0) = 1$ and, for $n \geq 1,$ 
\begin{equation*} \label{def:Theta}
\Theta_{\Gamma}(e_n) = \frac{(-1)^n}{n!} \mbox{NM}_{\Gamma,n}(1,y).
\end{equation*}
It follows that 
\begin{align*}
\Theta_{\Gamma}(H(t)) & = \sum_{n \geq 0} \Theta_{\Gamma}(h_n)t^n  = \frac{1}{\Theta_{\tau}(E(-t))} = \frac{1}{1 + \sum_{n \geq 1} (-t)^n \Theta_{\Gamma}(e_n)} \\
& = \frac{1}{1 + \sum_{n \geq 1} \frac{t^n}{n!} \mbox{NM}_{\Gamma,n}(1,y)} = 
U_{\Gamma}(t,y).
\end{align*}
By  (\ref{htoe}), we have \begin{eqnarray}\label{eq:basic1}
n! \Theta_{\Gamma}(h_n) &=& n! \sum_{\la \vdash n} (-1)^{n-\ell(\la)} 
B_{\la,n}~ \Theta_{\Gamma}(e_\la) \nonumber \\
&=& n! \sum_{\la \vdash n} (-1)^{n-\ell(\la)} \sum_{(b_1, \ldots, 
b_{\ell(\la)}) \in \mathcal{B}_{\la,n}} \prod_{i=1}^{\ell(\la)}  
\frac{(-1)^{b_i}}{b_i!} \mbox{NM}_{\Gamma,b_i}(1,y) \nonumber \\
&=& \sum_{\la \vdash n} (-1)^{\ell(\la)} \sum_{(b_1, \ldots, b_{\ell(\la)}) \in \mathcal{B}_{\la,n}} \binom{n}{b_1, \ldots, b_{\ell(\la)}}
\prod_{i=1}^{\ell(\la)}  \mbox{NM}_{\Gamma,b_i}(1,y).
\end{eqnarray}

Next, we want to give a combinatorial interpretation to the right hand side of (\ref{eq:basic1}). We select a brick tabloid $B = (b_1, b_2, \dots, b_{\ell(\la)} )$ of shape $(n)$ filled with bricks whose sizes induce the partition $\la$. We interpret the multinomial coefficient $\binom{n}{b_1, \ldots, b_{\ell(\la)}}$ as the number of ways to choose an ordered set partition 
$\mathcal{S} =(S_1, S_2, \dots, S_{\ell(\la)})$ of $\{1,2, \dots, n\}$ such that $|S_i| = b_i$ for $i = 1, \dots, \ell(\la).$ For each brick $b_i,$ we then fill the cells of $b_i$ with numbers from $S_i$ such that the entries in the brick reduce to a permutation $\sg^{(i)} = \sg_1 \cdots \sg_{b_i}$ in $\mathcal{NM}_{b_i}(\Gamma).$ We label each descent of $\sg$ that occurs within each brick as well as the last cell of each brick by $y.$  This accounts for the factor 
$y^{\des(\sg^{(i)})+1}$ within each brick. Finally, we use the factor $(-1)^{\ell(\la)}$ to change the label of the last cell of each brick from $y$ to $-y$. We will denote 
the filled labeled brick tabloid constructed in this 
way as $\langle B,\mathcal{S},(\sg^{(1)}, \ldots, \sg^{(\ell(\la))})\rangle$.  

For example, when $n = 17, \Gamma = \{1324, 1423, 12345\},$ and 
$B = (9,3,5,2),$ consider the ordered set partition $\mathcal{S}=(S_1,S_2,S_3,S_4)$ of $\{1,2,\dots, 17\}$, where $$S_1=\{ 2,5,6,9,11,15,16,17,19 \}, S_2 = \{7,8,14\},S_3 = \{1,3,10,13,18\},S_4 = \{4,12\},$$ and the permutations 
$\sg^{(1)} = 1~2~4~6~5~3~7~9~8 \in \mathcal{NM}_{9}(\Gamma),\sg^{(2)} = 1~3~2 \in \mathcal{NM}_{7}(\Gamma),\sg^{(3)} = 5~1~2~4~3 \in \mathcal{NM}_{5}(\Gamma),\mbox{~and~} \sg^{(4)} = 2~1 \in \mathcal{NM}_{2}(\Gamma).$
The construction of 
$\langle B,\mathcal{S},(\sg^{(1)}, \ldots, \sg^{(4)})\rangle$ is then
pictured in Figure \ref{fig:CFLBTab}.

\fig{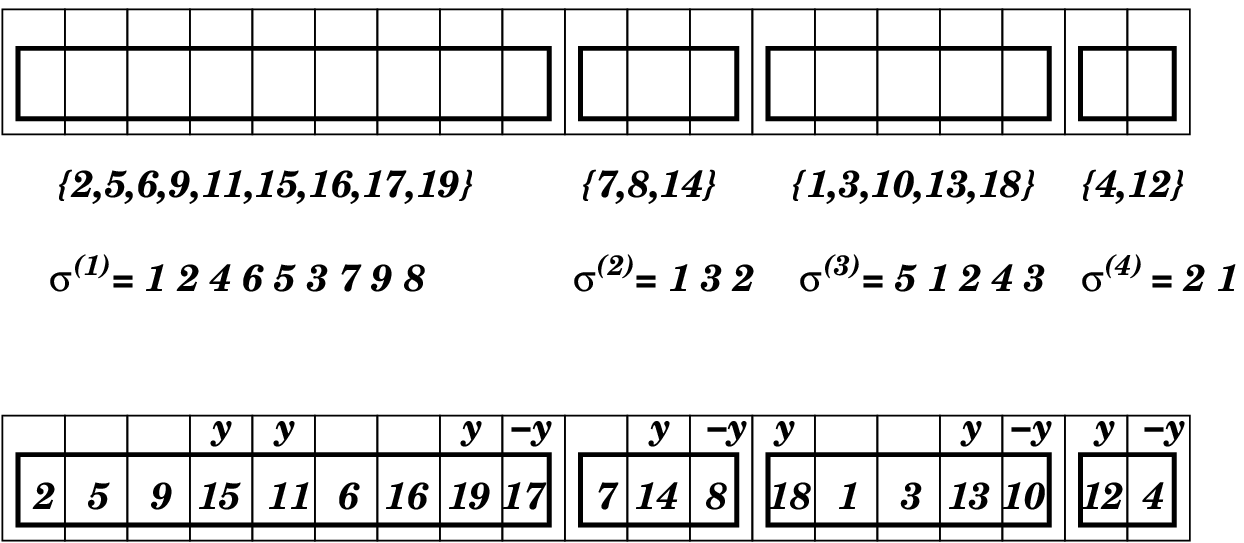}{The construction of a filled-labeled-brick tabloid.}

It is easy to see that we can recover 
the triple $ \langle B, (S_1, \dots, S_{\ell(\la)}), (\sg^{(1)}, \dots,\sg^{(\ell(\la))}  ) \rangle$ from $B$ and the permutation 
$\sg$ which is obtained by reading the entries in the 
cells from right to left.  We let $\mathcal{O}_{\Gamma, n}$ denote the set of all filled labeled brick tabloids created this way. That is, $\mathcal{O}_{\Gamma, n}$ consists of all pairs $O = (B, \sg )$ where 
\begin{enumerate}
\item $B = (b_1, b_2, \dots, b_{\ell(\la)})$ is a brick tabloid of shape $n$,  
\item $\sg = \sg_1  \cdots \sg_n$ is a permutation in $\mathfrak{S}_n$ such that there is no $\Gamma$-match of $\sg$ which 
lies entirely in a single brick of $B$, and 
\item if there is a cell $c$ such that a brick $b_i$ contains both cells $c$ and $c+1$ and $\sg_c > \sg_{c+1}$, then cell $c$ is labeled with a $y$ and the last cell of any brick is labeled with $-y$.
\end{enumerate}

We define the sign of each $O$ to be $sgn(O) = (-1)^{\ell(\la)}.$ The weight $W(O)$ of $O$ is defined to be the product of all the labels $y$ used in the brick. Thus, the weight of the filled labeled brick tabloid from Figure \ref{fig:CFLBTab} above is $W(O) = y^{11}.$ It follows that 
\begin{equation}\label{eq:basic2}
n!\Theta_{\Gamma}(h_n) = \sum_{O \in \mathcal{O}_{\Gamma,n}} sgn(O) W(O).
\end{equation}

Following \cite{JR2}, we next define a sign-reversing, weight-preserving involution $I: \mathcal{O}_{\Gamma, n} \rightarrow \mathcal{O}_{\Gamma, n} $. Given a filled labeled brick tabloid $(B,\sg) \in \mathcal{O}_{\Gamma, n} $ where 
$B=(b_1, \ldots, b_k)$, we read the cells of $(B,\sg)$ from left to right, looking for the first cell $c$ for 
which either 
\begin{enumerate}
\item[(i)] cell $c$ is labeled with a $y$, or 
\item[(ii)] cell $c$ is at the end of brick $b_i$ where $\sg_c > \sg_{c+1}$ and there is no $\Gamma$-match of $\sg$ that lies entirely in the cells of the bricks $b_i$ and $b_{i+1}$.  
\end{enumerate}

In case (i), we define $I_{\Gamma}(B,\sg)$ to be the filled labeled brick tabloid obtained from $(B,\sg)$ by breaking the brick $b_j$ that contains cell $c$ into two bricks $b_j'$ and $b_j''$ where $b_j'$ contains the cells of $b_j$ up to and including the cell $c$ while $b_j''$ contains the remaining cells of $b_j$. In addition, we change the labeling of cell $c$ from $y$ to $-y$. In case (ii), $I_{\Gamma}(B,\sg)$ is obtained by combining the two bricks $b_i$ and $b_{i+1}$ 
into a single brick $b$ and changing the label of cell $c$ from $-y$ to $y$. If neither case occurs, then we let $I_{\Gamma}(B,\sg) = (B,\sg)$. 

For instance, the image of the filled labeled brick tabloid from the Figure \ref{fig:CFLBTab} under this involution is shown below in Figure \ref{fig:CFLBTab1}.  


\fig{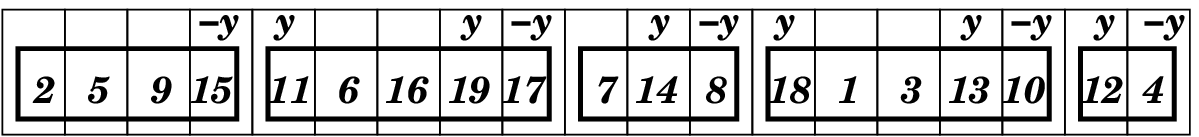}{$I_{\Gamma}(O)$ for $O$ in Figure \ref{fig:CFLBTab}.}

We claim that as long as each permutation in $\Gamma$ has at most 
one descent, then \begin{math}I_{\Gamma}\end{math} is an involution. 
Let $(B,\sg)$ be an element of $\mathcal{O}_{\gamma,n}$ which 
is not a fixed point of $I$. Suppose that $I(B,\sg)$ is defined 
using case (i) where we split a brick \begin{math}b_j\end{math} at 
cell \begin{math}c\end{math} which is labeled with a \begin{math}y\end{math}.  
In that case, we let 
\begin{math}a\end{math} be the number in cell \begin{math}c\end{math} and \begin{math}a^\prime\end{math} be the number
in cell \begin{math}c+1\end{math} which must also be in brick \begin{math}b_{j}\end{math}. Since cell $c$ is labeled with $y$, it must be the case 
that $a > a'$. Moreover, there can be no cell labeled \begin{math}y\end{math} that occurs before cell \begin{math}c\end{math} since otherwise 
we would not use cell \begin{math}c\end{math} to define $I(B,\sg)$. 
In this case, we must ensure that when we split \begin{math}b_j\end{math} into \begin{math}b_j^\prime\end{math} and \begin{math}b_j^{\prime \prime}\end{math}, 
we cannot combine the brick \begin{math}b_{j-1}\end{math} with  
\begin{math}b_j^\prime\end{math} because the number in that last cell of 
\begin{math}b_{j-1}\end{math} is greater than the number in the first cell of 
\begin{math}b_j^\prime\end{math} and there is no $\Gamma$-match 
in the cells of 
\begin{math}b_{j-1}\end{math} and  \begin{math}b_j^\prime\end{math} since 
in such a situation, $I_\Gamma(I_\Gamma(B,\sg)) \neq (B,\sg)$.   
However, since we always take an action on the leftmost cell possible 
when defining $I_\Gamma(B,\sg)$, we know that we cannot combine  \begin{math}b_{j-1}\end{math} and \begin{math}b_j\end{math} so that 
there must be a $\Gamma$-match in the cells of \begin{math}b_{j-1}\end{math} and \begin{math}b_j\end{math}. 
Moreover, if we could now combine bricks $b_{j-1}$ and $b_j^\prime$, then  
that $\Gamma$-match must have involved the number 
\begin{math}a^\prime\end{math} and 
the number in cell $d$ which is the last cell in brick $b_{j-1}$. But that is impossible because then there would be two descents among the numbers
between cell $d$ and cell $c+1$ which would violate our assumption 
that the elements of $\Gamma$  have at most one descent. Thus whenever we apply case (i) to 
define $I_\Gamma(B,\sg)$, 
the first action that we can take is to combine bricks 
\begin{math}b_j^\prime\end{math} and \begin{math}b_j^{\prime \prime}\end{math} so that \begin{math}I_{\Gamma}^2(B,\sg) = (B,\sg)\end{math}. 

If we are in case (ii), then again we can assume that there 
are no cells labeled \begin{math}y\end{math} that occur before cell \begin{math}c\end{math}.  When we combine 
brick \begin{math}b_i\end{math} and \begin{math}b_{i+1}\end{math}, then we will label cell \begin{math}c\end{math} with a \begin{math}y\end{math}. 
It is clear that combining the cells of \begin{math}b_i\end{math} and \begin{math}b_{i+1}\end{math} cannot 
help us combine the resulting brick \begin{math}b\end{math} with 
$b_{j-1}$ since, if there were a $\Gamma$-match that prevented us 
from combining bricks $b_{j-1}$ and $b_j$, then that same 
$\Gamma$-match will prevent us from combining $b_{j-1}$ and $b$.  
Thus, the first place where we can apply the involution 
will again be cell \begin{math}c\end{math} which is now labeled with a \begin{math}y\end{math} so that \begin{math}I_{\Gamma}^2(B,\sg) =(B,\sg)\end{math}.

It is clear that if $I_{\Gamma}(B,\sg) \neq (B,\sg)$, then $$sgn(B,\sg)W(B,\sg) = -sgn(I_{\Gamma}(B,\sg))W(I_{\Gamma}(B,\sg)).$$ Thus it follows from (\ref{eq:basic2}) that 
\begin{equation*}
n!\Theta_\Gamma(h_n) =  \sum_{O \in \mathcal{O}_{\Gamma,n}} \sgn{O} W(O) = 
\sum_{O \in \mathcal{O}_{\Gamma,n}, I_{\Gamma}(O) =O} \sgn{O} W(O). 
\end{equation*}
Hence if all permutations in $\Gamma$ have 
at most one descent, then 
\begin{equation}\label{eq:KEY}
U_{\Gamma,n}(y) = 
\sum_{O \in \mathcal{O}_{\Gamma,n}, I_{\Gamma}(O) =O} \sgn{O} W(O). 
\end{equation} 
Thus to compute $U_{\Gamma,n}(y)$, we must analyze the fixed points 
of $I_{\Gamma}$.  

If $(B,\sg)$ where $B =(b_1, \ldots, b_k)$ and $\sg = \sg_1 \cdots \sg_n$ 
is a fixed point of the involution $I_{\Gamma}$, then $(B,\sg)$ cannot have any cell labeled $y$ which means that the elements of $\sg$ that lie within any brick $b_j$ of $B$ must be increasing. If it is the case that an identity permutation $12 \cdots (k+1)$ is in $\Gamma$, then no brick of $B$ can have length greater than $k$. Next, consider any two consecutive bricks $b_i$ and $b_{i+1}$ in $B$.  Let $c$ be the last cell of $b_i$ and $c+1$ be the first cell of $b_{i+1}$.  Then either $\sg_c < \sg_{c+1}$ in which case we say there is an increase between bricks $b_i$ and $b_{i+1}$, or $\sg_c > \sg_{c+1}$ in which case we say there is a decrease between bricks $b_i$ and $b_{i+1}$. In the latter case, there must be a $\Gamma$-match of $\sg$ that lies in the cells of $b_i$ and $b_{i+1}$ which must necessarily involve $\sg_c$ and $\sg_{c+1}$. Finally, we claim that 
since all the permutations in $\Gamma$ start with 1, 
the minimal elements within the bricks of $B$ must increase from left to right. That is, consider two consecutive bricks $b_i$ and $b_{i+1}$ and let $c_i$ and $c_{i+1}$ be the first cells of $b_i$ and $b_{i+1},$ respectively. Suppose that $\sg_{c_i} > \sg_{c_{i+1}}$.  Let $d_i$ be the last cell of $b_i$. Then clearly $\sg_{c_{i+1}} < \sg_{c_i} \leq \sg_{d_i}$ so that there is a decrease between brick $b_i$ and brick $b_{i+1}$ and hence there must be a $\Gamma$-match of $\sg$ that lies in the cells of $b_i$ and $b_{i+1}$ that involves the elements of $\sg_{d_i}$ and $\sg_{c_{i+1}}$. But this is impossible since our assumptions ensure that $\sg_{c_{i+1}}$ is the smallest element that lies in the bricks $b_i$ and $b_{i+1}$ so that it can only play the role of $1$ in any $\Gamma$-match.  But since every element of $\Gamma$ starts with 1, then any $\Gamma$-match that lies in $b_i$ and $b_{i+1}$ that involves $\sg_{c_{i+1}}$ must lie entirely in brick $b_{i+1}$ which contradicts the fact that $(B,\sg)$ was a fixed point of $I_\Gamma$.

Thus, we have the following lemma describing the fixed points of the involution $I_{\Gamma}$.

\begin{lemma} \label{lem:keyGamma}
Let $\Gamma$ be a set of permutations which all start with 
1 and have at most one descent. 
Let $\mathbb{Q}(y)$ be the set of rational functions in 
the variable $y$ over the rationals $\mathbb{Q}$ and let 
$\Theta_\Gamma:\Lambda \rightarrow 
\mathbb{Q}(y)$ be the  ring homomorphism defined by setting  
$\Theta_{\Gamma}(e_0) =1$,  and  
$\Theta_{\Gamma}(e_n) = \frac{(-1)^n}{n!} \mbox{NM}_{\Gamma,n}(1,y)$ for $n \geq 1$.
 Then 
\begin{equation*}
n!\Theta_\Gamma (h_n) = 
\sum_{O \in \mathcal{O}_{\Gamma,n},I_\Gamma(O) = O}\sgn{O}W(O)
\end{equation*}
where $\mathcal{O}_{\Gamma,n}$ is the set of objects and 
$I_\Gamma$ is the involution defined above. Moreover, 
$O =(B,\sg) \in \mathcal{O}_{\Gamma,n}$ where $B=(b_1, \ldots, b_k)$ and 
$\sg = \sg_1 \cdots \sg_n$ is 
a fixed point of $I_\Gamma$ if and only if $O$ satisfies the following 
four properties: 
\begin{enumerate}
\item there are no cells labeled with $y$ in $O$, i.e., the 
elements in each brick of $O$ are increasing,
\item the first elements in each brick of $O$ form an increasing 
sequence, reading from left to right, 
\item if \begin{math}b_i\end{math} and \begin{math}b_{i+1}\end{math} are two consecutive bricks in \begin{math}B\end{math}, then 
either (a) there is increase between \begin{math}b_i\end{math} and \begin{math}b_{i+1}\end{math}, i.e., $\sg_{\sum_{j=1}^i b_j} < \sg_{1+\sum_{j=1}^i b_j}$, or (b) 
there is a decrease  between \begin{math}b_i\end{math} and \begin{math}b_{i+1}\end{math}, i.e., $\sg_{\sum_{j=1}^i b_j} > \sg_{1+\sum_{j=1}^i b_j}$,
and  there is a 
$\Gamma$-match contained in the elements of the cells of \begin{math}b_i\end{math} and \begin{math}b_{i+1}\end{math} which must necessarily involve 
$\sg_{\sum_{j=1}^i b_j}$ and  $\sg_{1+\sum_{j=1}^i b_j}$, and 
\item if $\Gamma$ contains an identity permutation 
$12 \cdots (k+1)$, then $b_i \leq k$ for all $i$. 
\end{enumerate}
\end{lemma}

Note that since $U_{\Gamma,n}(y) = n!\Theta_\Gamma (h_n)$, 
Lemma \ref{lem:keyGamma} gives us a combinatorial interpretation 
of $U_{\Gamma,n}(y)$. Since the weight of 
of any fixed point $(B,\sg)$ of $I_\Gamma$ is $-y$ raised to 
the number of bricks in $B$, it follows that 
$U_{\Gamma,n}(-y)$ is always a polynomial 
with non-negative integer coefficients.  We will exploit this combinatorial 
interpretation to prove the main results of this paper.

\section{Proof of Theorem \ref{thm:1-2-gen}}

Let $k_1,k_2 \geq 2$ and $p= k_1+k_2$. We consider the family of permutations $\Gamma = \Gamma_{k_1,k_2}$ in $\mathfrak{S}_p$ where 
$$\Gamma_{k_1,k_2} = \{\sigma \in \mathfrak{S}_p: \sigma_1=1, \sigma_{k_1+1}=2, \sigma_1 < \sigma_2< \cdots<\sigma_{k_1}~ \&~\sigma_{k_1+1} < \sigma_{k_1+2}< \cdots<\sigma_{p} \}.$$ 
We start this section by giving a proof of Theorem 
\ref{thm:1-2-gen}. At the end of this section, we shall 
consider how to 
compute $U_{\Gamma_{k_1,k_1,s}}(y,t)$ where 
$$ \Gamma_{k_1,k_1,s} = \Gamma_{k_1,k_1} \cup \{12 \cdots s (s+1)\}.$$

By (\ref{eq:KEY}), we must show that the coefficients 
$$U_{\Gamma,n}(y) = \sum_{O \in \mathcal{O}_{\Gamma,n}, I_{\Gamma}(O) =O} \sgn{O} W(O)$$
have the following properties:
\begin{enumerate} 
\item  $U_{\Gamma,1}(y)=-y$, and 
\item for $n>1$, $U_{\Gamma,n}(y) = (1-y)U_{\Gamma,n-1}(y) -y\binom{n-2}{k_1-1}\left( U_{\Gamma,n-M}(y) +y\sum_{i=1}^{m-1}U_{\Gamma,n-M-i}(y) \right),$ where $m = \min\{k_1, k_2\}$ and $M = \max\{k_1,k_2\}.$
\end{enumerate}

We will divide the proof into two cases, one where 
$k_1 \geq k_2$ and the other where $k_1 < k_2$. \\
\ \\
{\bf Case I.} $k_1 \geq k_2$.\\

Let $(B,\sg)$ be a fixed point of $I_{\Gamma}$ where 
$B=(b_1, \ldots, b_k)$ and $\sg=\sg_1 \cdots \sg_n$. 
We know that 1 is in the first cell of  $(B,\sg)$. We claim that 2 must be in cell 2 or cell $k_1+1$ of  $(B,\sg)$.  To see this, suppose that 2 is in cell $c$ where $c \ne 2$ and $c\ne k_1+1.$ Since there is no descent within any brick, 2 must be the first cell of its brick. Moreover, since the minimal elements 
of the bricks form an increasing sequence, reading from left to right, 
2 must be in the first cell of the second brick $b_2$. 
Thus, 1 is in the first cell of the first brick $b_1$ and 2 is in the first cell of the second brick $b_2$. Since $c > 2$, there is a decrease 
between bricks $b_1$ and $b_2$ and, hence, there must be 
a $\Gamma$-match of $\sg$ contained cells of $b_1$ and $b_2$ 
which involves 2 and the last cell of $b_1$.  Since 
all the elements of $\Gamma$ start with 1, this $\Gamma$-match 
must also involve 1 since only 1 can play the role of 1 in a 
$\Gamma$-match that involves 2 and the last cell of $b_1$. 
But in all such $\Gamma$-matches, 2 must be in cell $k_1+1$. 
Since $c \neq k_1 +1$, this means that there can 
be no $\Gamma$-match contained in the cells of $b_1$ and 
$b_2$ which contradicts the fact that $(B,\sg)$ is a 
fixed point of $I_\Gamma$. 
 
Thus, we have two subcases.

\ \\
{\bf Subcase 1.} 2 is in cell 2 of $(B,\sg)$.\\
\ \\
In this case there are two possibilities, namely, either 
(i) 1 and 2 are both in the first brick $b_1$ of $(B,\sg)$ or (ii) 
brick $b_1$ is a single cell filled with 1 and 2 is in the first cell of the second brick $b_2$ of 
$(B,\sg)$.  In either case, we know that 1 is not part of a 
$\Gamma$-match in $(B,\sg)$. So if we remove cell 1 from $(B,\sg)$ and subtract $1$ from the elements in the remaining cells, we will obtain a fixed point $O'$ of $I_{\Gamma}$ in $\mathcal{O}_{\Gamma,n-1}.$

Moreover, we can 
create a fixed point $O=(B,\sg) \in \mathcal{O}_n$ satisfying conditions 
(1), (2), (3)  and (4) of Lemma \ref{lem:keyGamma} where 
$\sg_2 =2$ by starting with a fixed point $(B',\sg') \in 
\mathcal{O}_{\Gamma,n-1}$ of $I_\Gamma$, where $B' =(b_1', \ldots, b_r')$ and 
$\sg' =\sg_1' \cdots \sg_{n-1}'$, and then letting 
 $\sg = 1 (\sg_1'+1) \cdots (\sg_{n-1}' +1)$, and 
setting  $B = (1,b_1', \ldots, b_r')$ 
or setting $B = (1+b_1', \ldots, b_r')$.

It follows that fixed points in Case 1 will contribute 
$(1-y)U_{\Gamma,n-1}(y)$ to 
$U_{\Gamma,n}(y)$.\\
\ \\
{\bf Subcase 2.} 2 is in cell $k_1+1$ of $(B,\sg)$. \\
\ \\
Since there is no decrease within the bricks of $(B,\sg)$ and the first numbers of the bricks are increasing, reading from left to right, it must be the case that 2 is in the first cell of $b_2.$  Thus $b_1$ has exactly $k_1$ cells. In addition, $b_2$ has at least $k_2$ cells since otherwise, there could be no $\Gamma$-match contained in the cells of $b_1$ and $b_2$ and we could combine the bricks $b_1$ and $b_2$, which would mean that $(B,\sg)$ is not a fixed point of $I_{\Gamma}.$ By our argument above, it must be the 
case that the $\Gamma$-match of $\sg$ contained in the cells 
of $b_1$ and $b_2$ must start in the first cell. 
We first choose $k_1-1$ numbers to fill in the remaining cells of $b_1.$ There are \(\binom{n-2}{k_1-1}\) ways to do this. For each such choice, 
we let $O'$ be the result by removing the first $k_1$ cells from $(B,\sg)$ 
and replacing the $i^{th}$ largest remaining number by $i$ for 
$i=1, \ldots, n-k_1$, then $O'$ will be a fixed point in $\mathcal{O}_{\Gamma,n-k_1}$ whose first brick is of size 
greater than or equal to $k_2$.

On the other hand, suppose that we 
 start with $O' \in \mathcal{O}_{\Gamma,n-k_1}$ 
which is a fixed point of $I_\Gamma$ and whose first brick is of size greater than or equal to $k_2$. Then we can 
take any $k_1-1$ numbers $1< a_1 < a_2 < \cdots < a_{k_1-1} \leq n$ 
and add a new brick at the start which contains 
$1,a_1, \ldots a_{k_1-1}$ followed by $O''$ which is the result  
of replacing the numbers in $O'$ by the numbers in $\{1, \ldots, n\} - \{1,a_1, \ldots a_{k_1-1}\}$ maintaining the same relative order, then we will create a fixed point $O$ of $I_\Gamma$ of size $n$ whose first brick is of size $k_1$ and whose second brick starts with 2. 

Thus we need to count the number of fixed points in $\mathcal{O}_{\Gamma,n-k_1}$whose first brick has size at least $k_2$. 
Suppose that $V=(D,\tau)$ is a fixed point 
of $\mathcal{O}_{\Gamma,n-k_1}$ where $D=(d_1, \ldots, d_k)$ and 
$\tau = \tau_1 \cdots \tau_{n-k_1}$.  Now 
if $d_1 = j < k_2$, then there cannot be a decrease 
between bricks $d_1$ and $d_2$ because otherwise 
there would have been a $\Gamma$-match starting at cell 
1 contained in the bricks $d_1$ and $d_2$ which is impossible 
since all permutations in $\Gamma$ have their only descent 
at position $k_1 > j$.  This means 
that the first brick $d_1$ must be filled with $1, \ldots j$.  That 
is, since the minimal elements of the bricks are 
increasing reading from left to right, we must 
have that the first element of $d_2$, namely $\tau_{j+1}$, 
is less than all the elements to its right and we 
have shown that all the elements in the first brick 
are less than $\tau_{j+1}$. It follows 
that $\tau_1 \cdots \tau_{j+1} = 12 \cdots j(j+1)$. 
Therefore, if we let $V'$ be the result of removing the entire first brick of $V$ and subtracting $j$ from the remaining numbers, then $V'$ is a fixed point in  $\mathcal{O}_{\Gamma,n-k_1-j}.$

It follows that 
\begin{equation*} 
U_{\Gamma,n-k_1}(y) - \sum_{j=1}^{k_2-1}(-y)U_{\Gamma,n-k_1-j}(y)
\end{equation*}
equals the sum over all fixed points 
of $I_{\Gamma,n-k_1}$ whose first brick has size at least $k_2$. 
Hence the contribution of fixed points in Case 2 to $U_{\Gamma,n}(y)$ is
\begin{equation*}
 (-y)\binom{n-2}{k_1-1}\left(U_{\Gamma,n-k_1}(y) + \sum_{j=1}^{k_2-1}y U_{\Gamma,n-k_1-j}(y)\right).
\end{equation*}

Combining the two cases, we see that for $n > 1$, 
\begin{equation}\label{eq:fincaseI}
U_{\Gamma,n}(y) = (1-y)U_{\Gamma,n-1}(y) -y\binom{n-2}{k_1-1}\left( U_{\Gamma,n-k_1}(y) +y\sum_{j=1}^{k_2-1}U_{\Gamma,n-k_1-i}(y) \right).
\end{equation}
\ \\
{\bf Case II.} $k_1 < k_2$. \\

Let $O =(B,\sg)$ be a fixed point of $I_{\Gamma}$ where 
$B=(b_1, \ldots, b_k)$ and $\sg = \sg_1 \cdots \sg_n$. 
We know that 1 is in the first cell of  $O$. 
By the same argument as in Case I, 
we know that 2 must be in cell 2 or cell $k_1+1$ of  $O$. We 
now consider two cases depending on the position of 2 in $O$. \\
\ \\
{\bf Subcase A.} 2 is in cell 2 of $(B,\sg)$.\\
\ \\
By the same argument that 
we used in Subcase 1 of Case I, we can conclude that the fixed points 
of $I_\Gamma$ in 
Subcase A will contribute $(1-y)U_{\Gamma,n-1}(y)$ to $U_{\Gamma,n}(y)$.\\
\ \\
{\bf Subcase B.} 2 is in cell $k_1+1$ of $(B,\sg)$. \\
\ \\
Since the minimal elements of the bricks are increasing, reading from left to right, it must be the case that 2 is in the first cell of $b_2.$ Thus, $b_1$ has exactly $k_1$ cells, $b_2$ has at least $k_2$ cells, and there is a 
$\Gamma_{k_1,k_2}$-match in the cells of $b_1$ and $b_2$ which 
must start at cell 1. 

We first choose $k_1-1$ numbers to fill in the remaining cells of $b_1$. 
There are \(\binom{n-2}{k_1-1}\) ways to do this. For each of such choice, 
let $d_1< \cdots < d_{k_2-k_1-1}$ be the smallest $k_2-k_1-1$ numbers in $\{1,2, \dots,n\}-\{\sg_1, \ldots, \sg_{k_1+1}\}$. 
We claim that it must be the case that 
$\sg_{k_1+1+i} = d_i$ for $i = 1, \ldots, k_2-k_1-1$.  If not, 
let $j$ be the least $i$ such that $\sg_{k_1+1+i} \neq d_i$. 
Then $d_i$ cannot be in brick $b_2$ so that it must be 
the first element in brick $b_3$.  But 
then there will be a decrease between bricks $b_2$ and $b_3$ which means that there must be a $\Gamma_{k_1,k_2}$-match contained in 
the cells of $b_2$ and $b_3$.  Note that there is only one 
descent in each permutation of $\Gamma_{k_1,k_2}$ and 
this descent must occur at position $k_1$. It follows 
that this $\Gamma_{k_1,k_2}$-match must start at 
the $(k_2-k_1)^{th}$ cell of $b_2$.  But this is impossible 
since our assumption will ensure 
that $\sg_{k_1+1+(k_2-k_1-1)} = \sg_{k_2} > d_i$.

It then follows that if we let $O'$ be the result by removing the first $k_2$ cells from $O$ and adjusting the remaining numbers in the cells, then $O'$ will be a fixed point in $\mathcal{O}_{\Gamma,n-k_2}$ that starts with at least $k_1$ cells in the first brick. Then we can argue exactly as we did 
in Subcase 2 of Case I the contribution of fixed points in Case B to $U_{\Gamma,n}(y)$ is 
\begin{equation*}
-y \binom{n-2}{k_1-1}\left(U_{\Gamma,n-k_2}(y) +\sum_{j=1}^{k_1-1}yU_{\Gamma,n-k_2-j}(y)\right).
\end{equation*}
\ \\It follows that in Case II
\begin{equation}\label{eq:fincaseII}
U_{\Gamma_{k_1,k_2},n}(y) = (1-y) U_{\Gamma_{k_1,k_2},n}(y) -y \binom{n-2}{k_1-1}\left(U_{\Gamma,n-k_2}(y) +\sum_{j=1}^{k_1-1}yU_{\Gamma,n-k_2-j}(y)\right)
\end{equation}  
for $n > 1$. 
\ \\

Comparing equations (\ref{eq:fincaseI}) and (\ref{eq:fincaseII}), 
it is easy to see that if $m=\min(k_1,k_2)$ and 
$M = \max(k_1,k_2)$, then  
\begin{align*}
\displaystyle   U_{\Gamma_{k_1,k_2},n}(y) &= (1-y)U_{\Gamma_{k_1,k_2},n-1}(y) -y\binom{n-2}{k_1-1}\left( U_{\Gamma,n-M}(y) +y\sum_{i=1}^{m-1}U_{\Gamma,n-M-i}(y) \right) 
\end{align*}
for all $n >1$ which proves Theorem \ref{thm:1-2-gen}.

For example, consider the special case where $k_1 = k_2 = 2$. 
Then by Corollary \ref{cor:13241423}, 
\begin{align*} 
U_{\Gamma_{2,2},n}(y)= (1-y)U_{\Gamma_{2,2},n-1}(y) - 
y(n-2)\left( U_{\Gamma_{2,2},n-2}(y) + yU_{\Gamma_{2,2},n-3}(y)\right).\end{align*}

In Table \ref{tab:U1-2-}, we computed $U_{\Gamma_{2,2},n}(y)$ for $n \leq 14$. 

\begin{table}[h]
\begin{center}
\begin{tabular}{c|l}
n & $U_{\Gamma_{2,2},n}(-y)$ \\ 
\hline 1 & $y$\\
2 & $y + y^2$ \\
3 & $y + 2y^2 + y^3$\\
4 & $y + 5y^2 +3y^3 + y^4$\\
5 & $y + 9y^2 + 11y^3 + 4y^4 + y^5$\\
6 & $y+14y^2+36y^3 + 19y^4 +5y^5 + y^6$ \\ 
7 & $y+20y^2+90y^3 + 85y^4 +29y^5 + 6y^6 + y^7$ \\
8 & $y+27y^2+188y^3 + 337y^4 +162y^5 + 41y^6 + 7y^7 + y^8$ \\ 
9 & $y+35y^2+348y^3 + 1057y^4 +842y^5 + 273y^6 + 55y^7 + 8y^8 + y^9 $  \\ 
10 & $y+44y^2+591y^3 + 2749y^4 + 3875y^5  + 1731y^6 + 424y^7 + 71y^8 + 9y^9 +y^{10}$  \\ 
11 & $y+54y^2+941y^3 + 6229y^4 + 14445y^5 + 10151y^6 + 3154y^7 + 621y^8 $  \\
& \hspace{290pt} $+ 89y^9+ 10y^{10} + y^{11}$\\ 
12 & $y + 65 y^2 + 1425 y^3 + 12730 y^4 + 44684 y^5 + 52776 y^6 + 
 22195 y^7 + 5285 y^8 $ \\
 &\hspace{240pt} $+ 870 y^9 + 109 y^{10} + 11 y^{11} + y^{12}   $\\
13 & $y + 77 y^2 + 2073 y^3 + 24022 y^4 + 119432 y^5 + 226116 y^6 + 144007 y^7 + 43133 y^8$\\
 & \hspace{185pt} $ + 8322 y^9  +1177 y^{10} + 131 y^{11} + 12 y^{12} + y^{13}$\\
14 & $y + 90 y^2 + 2918 y^3 + 42547 y^4 + 284922 y^5 + 807008 y^6 + 830095 y^7 + 331668 y^8   $  \\
  &  \hspace{130pt} $77027 y^9 + 12487 y^{10} + 1548 y^{11} + 155 y^{12} + 13 y^{13} + y^{14}  $ 
\end{tabular} 
\end{center}
\caption{The polynomials $U_{\Gamma_{2,2},n}(-y)$ for $\Gamma_{2,2} = \{1324,1423\}$} \label{tab:U1-2-}
\end{table}

We observe that the polynomials 
$U_{\Gamma_{2,2},n}(-y)$ in Table \ref{tab:U1-2-} are all log-concave. Here, a polynomial $P(y)=a_0+a_1y+\dots+a_ny^n$ is called \emph{log-concave} if $a_{i-1}a_{i+1}<a_i^2,$ for all $i=2,\dots, n-1$, and it is called \emph{unimodal} if there exists an index $k$ such that $a_i\le a_{i+1}$ for $1\leq i \leq k-1$ and $a_i\ge a_{i+1}$ for $k\leq i \leq n-1$. We conjecture that the polynomials $U_{\Gamma_{2,2}, n}(-y)$ are log-concave and, hence, unimodal for all $n$. 
We checked this holds for $n \leq 21$. 

One might 
hope to prove the unimodality of the polynomials $U_{\Gamma_{2,2}, n}(-y)$ 
by using the recursion 
\begin{equation}\label{unim}
U_{\Gamma_{2,2}, n}(-y) = (1+y)U_{\Gamma_{2,2}, n-1}(-y) +
(n-2)y U_{\Gamma_{2,2}, n2}(-y) + (n-2)y^2U_{\Gamma_{2,2}, n-3}(-y)
\end{equation}
and showing that for large enough $n$, the polynomials on the 
right hand side of (\ref{unim}) are all unimodal polynomials 
whose maximum coefficients occur at the same power of $y$.  There are 
two problems with this idea. First, assuming that 
$U_{\Gamma_{2,2}, n}(-y)$ is a unimodal polynomial whose maximum coefficient 
occurs that $y^j$, then we know that 
$(1+y)U_{\Gamma_{2,2}, n}(-y)$ is a unimodal polynomial. However, 
it could be that the maximum coefficient of 
$(1+y)U_{\Gamma_{2,2}, n}(-y)$ occurs at $y^j$ or at $y^{j+1}$. 
That is, if $P(y)$ is a unimodal 
polynomial whose maximum coefficient occurs at 
$y^k$, then $(1+y)P(y)$ could have its maximum coefficient occur 
at either $y^k$ or $y^{k+1}$. For example, 
$$
(1+y)(1+5y+2y^2) = 1+6y+7y^2+2y^3
$$
while 
$$
(1+y)(2+5y+y^2) =  2+7y+6y^2+y^3.
$$
Thus where the maximum coefficient of $(1+y)U_{\Gamma_{2,2}, n}(-y)$ occurs 
depends on the relative values of the coefficients on either side 
of the maximum coefficient of $U_{\Gamma_{2,2}, n}(-y)$.
For $n \leq 20$, the maximum coefficient of 
$(1+y)U_{\Gamma_{2,2}, n}(-y)$ occurs at the same power of $y$ 
where the maximum coefficient of $U_{\Gamma_{2,2}, n}(-y)$ occurs, 
but it is not obvious that this holds for all $n$. 

Second, it is not clear where to conjecture 
the maximum coefficients in the polynomials occur. That is, 
one might think from the table that for $n \geq 6$, the maximum 
coefficient in $U_{\Gamma_{2,2}, n}(-y)$ occurs at $y^{\Floor[n/2]+1}$,
but this does not hold up. For example, the maximum coefficient 
$U_{\Gamma_{2,2}, 18}(-y)$ occurs at $y^8$ and the maximum coefficient 
$U_{\Gamma_{2,2}, 19}(-y)$ occurs at $y^9$. Moreover,  
the maximum coefficient 
$U_{\Gamma_{2,2}, 26}(-y)$ occurs at $y^{12}$ and the maximum coefficient 
$U_{\Gamma_{2,2}, 27}(-y)$ occurs at $y^{12}$. Thus it is not 
clear how to use the recursion (\ref{unim}) to even prove the 
unimodality of the polynomials $U_{\Gamma_{2,2},n}(-y)$ much less 
prove that such polynomials are log concave.

When $k_1$ is larger than $k_2$, the polynomials $U_{\Gamma_{k_1,k_2},n}(-y)$ are not always unimodal. For example, consider the case where $k_1 = 6$ and $k_2 = 4$. Mathematica once again allows us to compute $U_{\Gamma_{6,4}, n}(-y)$ for $n = 10$ and $11.$ It is quite easy to see from Table \ref{tab:U64} that neither polynomial is unimodal.

\begin{table}[h]
\begin{center}
\begin{tabular}{c|l}
n & $U_{\Gamma_{6,4},n}(-y)$ \\ 
\hline
10 & $y+65y^2+36y^3 + 84y^4 +126y^5 + 126y^6 + 84y^7 + 36y^8 + 9y^9 +y^{10}$  \\ 
11 & $y+192y^2+227y^3 + 120y^4 + 210y^5  + 252y^6 + 210y^7+ 120y^8 + 45y^9 + 10y^{10} + y^{11}$  \\ 
\end{tabular} 
\end{center}
\caption{The polynomials $U_{\Gamma_{6,4},n}(-y)$} \label{tab:U64}
\end{table}

\subsection{Adding an identity permutation to $\Gamma_{k_1,k_2}$}

In this subsection, we want to consider the effect 
of adding an identity permutation to $\Gamma_{k_1,k_2}$. 
To simplify our analysis, we shall consider 
only the case where $k_1 =k_2$, but the same type 
of analysis can be carried out in general. Thus 
assume that  
$s \geq k_1 = k_2 \geq 2$ and let 
$\Gamma_{k_1,k_1,s} = \Gamma_{k_1,k_1} \cup \{12 \cdots s(s+1)\}$. 
Then we know that 
$$U_{\Gamma_{k_1,k_1,s},n}(y) = 
\sum_{O \in \mathcal{O}_{\Gamma_{k_1,k_1,s},n},\ I_{\Gamma_{k_1,k_1,s}}(O) =O} 
\sgn{O}W(O).$$

We want to classify the fixed points 
of $I_{\Gamma_{k_1,k_1,s}}$ by the size of the first brick. 
By Lemma \ref{lem:keyGamma}, it must be the case 
that the size of the first brick is less than or equal 
to $s$. We let 
$U^{(r)}_{\Gamma_{k_1,k_1,s},n}(y)$ denote the sum of 
$\sgn{O}W(O)$ over all fixed points of $I_{\Gamma_{k_1,k_1,s}}$ 
whose first brick is of size $r$.  
Thus,
\begin{equation}\label{eq:k1s0}
U_{\Gamma_{k_1,k_1,s},n}(y)= 
\sum_{r=1}^s U^{(r)}_{\Gamma_{k_1,k_1,s},n}(y).
\end{equation}
Now let 
$O=(B,\sg)$ be a fixed point of $I_{\Gamma_{k_1,k_1,s}}$ 
where $B=(b_1, \ldots, b_k)$ and $\sg =\sg_1 \cdots \sg_n$. 
By our arguments above, if $b_1 < k_1$, then the elements 
in the first brick of $(B,\sg)$ are $1, \ldots ,b_1$ so that  
for $1 \leq r < k_1$, 
\begin{equation}\label{eq:k1s1}
U^{(r)}_{\Gamma_{k_1,k_1,s},n}(y) = 
-y U_{\Gamma_{k_1,k_1,s},n-r}(y).
\end{equation}
Let 
$$U^{(\geq k_1)}_{\Gamma_{k_1,k_1,s},n}(y) =
\sum_{r=k_1}^s U^{(r)}_{\Gamma_{k_1,k_1,s},n}(y)$$ 
be the sum of $\sgn{O}W(O)$ over all fixed points 
of $I_{\Gamma_{k_1,k_1,s}}$ whose first brick has 
size greater than or equal to $k_1$.  Clearly, 
\begin{eqnarray*}
U_{\Gamma_{k_1,k_1,s},n}(y)&=& U^{(\geq k_1)}_{\Gamma_{k_1,k_1,s},n}(y) + 
\sum_{r=1}^{k_1-1} U^{(r)}_{\Gamma_{k_1,k_1,s},n}(y) \\
&=& U^{(\geq k_1)}_{\Gamma_{k_1,k_1,s},n}(y) + 
\sum_{r=1}^{k_1-1} (-y)U_{\Gamma_{k_1,k_1,s},n-r}(y)
\end{eqnarray*}
so that 
\begin{equation}\label{eq:k1s2}
U^{(\geq k_1)}_{\Gamma_{k_1,k_1,s},n}(y) = 
U_{\Gamma_{k_1,k_1,s},n}(y) + 
\sum_{r=1}^{k_1-1} y U_{\Gamma_{k_1,k_1,s},n-r}(y).
\end{equation}

Now suppose that $r > k_1$. Then we claim 
that $\sg_i =i$ for $i =1, \ldots, r-k_1+1$.  That is, 
we know that $\sg_1 =1$ so that if it is not the case that $\sg_i =i$ for $i =1, \ldots, r-k_1+1,$ there must be a least $i \leq r-k_1+1$ which is not 
in the first brick of $(B,\sg)$.  Since there are no 
descents of $\sg$ within bricks and the minimal elements 
of the bricks of $(B,\sg)$ are increasing, reading from 
left to right,  it 
must be that $i$ is the first element of brick $b_2$ and 
there is a decrease between bricks $b_1$ and $b_2$.  Thus 
there is a $\Gamma_{k_1,k_1,s}$-match that lies in 
the cells of $b_1$ and $b_2$ and the only place that 
such a match can start is at cell $r-k_1+1$.  But this is impossible  
since we would have $\sg_{r-k_1+1} > i$ which is incompatible 
with having a  $\Gamma_{k_1,k_1,s}$-match starting at cell $r-k_1+1$. 
It follows that we can remove the first $r-k_1$ elements from  
$(B,\sg)$ and reduce the remaining elements by $r-k_1$ to  
produce a fixed point of $I_{\Gamma_{k_1,k_1,s}}$ of size 
$n-(r-k_1)$ whose first brick has size $k_1$. Vice versa, if we start with a fixed point 
$ (D,\tau)$ of $I_{\Gamma_{k_1,k_1,s}}$ of size $n-(r-k_1)$ where 
$D=(d_1, \ldots, d_k)$, $\tau = \tau_1 \cdots \tau_{n-(r-k_1)}$, and 
$d_1 =k_1$,  
then if we add $1, \ldots, r-k_1$ to the first brick and raise 
the remaining numbers by $r-k_1$, we will produce a fixed 
point of $I_{\Gamma_{k_1,k_1,s}}$ whose first brick is of size $r$. 
It follows that for $k_1 < r \leq s$, 
\begin{equation}\label{eq:k1s3}
U^{(r)}_{\Gamma_{k_1,k_1,s},n}(y) = 
U^{(k_1)}_{\Gamma_{k_1,k_1,s},n-(r-k_1)}(y).
\end{equation}
Thus 
\begin{equation}\label{eq:k1s4}
U^{(\geq k_1)}_{\Gamma_{k_1,k_1,s},n}(y) =
\sum_{p=0}^{s-k_1} U^{(k_1)}_{\Gamma_{k_1,k_1,s},n-p}(y).
\end{equation}

Finally consider 
$U^{(k_1)}_{\Gamma_{k_1,k_1,s},n}(y)$. Let $(B,\sg)$ be 
a fixed point of $I_{\Gamma_{k_1,k_1,s}}$ where 
$B=(b_1, \ldots, b_k)$, $b_1 =k_1$, and $\sg = \sg_1 \cdots \sg_n$. 
We then 
have two cases.\\
\ \\
{\bf Case 1.}  2 is in brick $b_1$. \\

In this case, we claim that 
the first brick must contain the elements $1, \ldots, k_1$. That is, 
in such a situation 1 cannot be involved in a $\Gamma_{k_1,k_1,s}$-match 
in $\sg$ which means that there is not enough room for a $\Gamma_{k_1,k_1,s}$-match that involves any elements from 
the first brick. Thus as before, we can remove the first 
brick from $(B,\sg)$ and subtract $k_1$ from the remaining elements 
of $\sg$ to produce a fixed point $(D,\tau)$ of 
$I_{\Gamma_{k_1,k_1,s}}$ of size $n-k_1$.  Such fixed points 
contribute   $(-y)U_{\Gamma_{k_1,k_1,s},n-k_1}(y)$ to 
$U^{(k_1)}_{\Gamma_{k_1,k_1,s},n}(y)$.\\ 
\ \\
{\bf Case 2.} 2 is in brick $b_2$. \\

In this case, we can argue as above that $2$ be the first 
cell of the second brick $b_2$ and $b_2$ starts at cell $k_1+1$. 
Then we have $\binom{n-2}{k_1-1}$ ways to choose the remaining 
elements in the first brick and if we remove the first brick 
and adjust the remaining elements, we will produce a fixed 
point $(D,\tau)$ of $I_{\Gamma_{k_1,k_1,s}}$ of size $n-k_1$  
whose first brick is of size greater than or equal to $k_1$. 
Such fixed points 
contribute   
$(-y)\binom{n-2}{k_1-1}U^{(\geq k_1)}_{\Gamma_{k_1,k_1,s},n-k_1}(y)$ to 
$U^{(k_1)}_{\Gamma_{k_1,k_1,s},n}(y)$.
 
It follows that 
\begin{eqnarray}\label{eq:k1s5}
U^{(k_1)}_{\Gamma_{k_1,k_1,s},n}(y) &=& 
-yU_{\Gamma_{k_1,k_1,s},n-k_1}(y)-
y\binom{n-2}{k_1-1}U^{(\geq k_1)}_{\Gamma_{k_1,k_1,s},n-k_1}(y)  \nonumber \\
&=& -yU_{\Gamma_{k_1,k_1,s},n-k_1}(y)-  \nonumber \\
&&y\binom{n-2}{k_1-1}\left(U_{\Gamma_{k_1,k_1,s},n-k_1}(y)+y\sum_{r=1}^{k_1-1}
U_{\Gamma_{k_1,k_1,s},n-k_1-r}(y)\right).
\end{eqnarray}

Putting equations (\ref{eq:k1s0}), 
(\ref{eq:k1s1}), (\ref{eq:k1s2}), (\ref{eq:k1s3}), 
(\ref{eq:k1s4}), and (\ref{eq:k1s5}) together, 
we see that 
\begin{align*}
& U_{\Gamma_{k_1,k_1,s},n}(y)  \\
& = -y \sum_{r=1}^{k_1-1} U_{\Gamma_{k_1,k_1,s},n-r}(y) + \sum_{p=0}^{s-k_1} U^{(k_1)}_{\Gamma_{k_1,k_1,s},n-p}(y)  \\
& = -y \sum_{r=1}^{k_1-1} U_{\Gamma_{k_1,k_1,s},n-r}(y) - y \sum_{p=0}^{s-k_1} U_{\Gamma_{k_1,k_1,s},n-p-k_1}(y) \\
&\hspace{100 pt} +\binom{n-p-2}{k_1-1} \left(U_{\Gamma_{k_1,k_1,s},n-p-k_1}(y)+y 
\sum_{a=1}^{k_1-1} U_{\Gamma_{k_1,k_1,s},n-p-k_1-a}(y)\right)\\
& = -y \sum_{r=1}^{k_1-1} U_{\Gamma_{k_1,k_1,s},n-r}(y) - y \Bigg(\sum_{p=0}^{s-k_1}\left(1 + \binom{n-p-2}{k_1-1}\right)U_{\Gamma_{k_1,k_1,s},n-p-k_1}(y)\\
&\hspace{210 pt}  +y \binom{n-p-2}{k_1-1} \sum_{a=1}^{k_1-1} U_{\Gamma_{k_1,k_1,s},n-p-k_1-a}(y)\Bigg).\end{align*}

Thus we have the following theorem. 

\begin{theorem} Let $\Gamma_{k_1,k_1,s} = 
\Gamma_{k_1,k_1} \cup \{12 \cdots s(s+1)\}$ where $s \geq k_1$. 
Then $U_{\Gamma_{k_1,k_1,s},1}(y)=-y$ and for $n \geq 2$, 
\begin{multline*}
U_{\Gamma_{k_1,k_1,s},n}(y) =-y \sum_{r=1}^{k_1-1} U_{\Gamma_{k_1,k_1,s},n-r}(y)-y \Bigg(\sum_{p=0}^{s-k_1}\left(1 + \binom{n-p-2}{k_1-1}\right)U_{\Gamma_{k_1,k_1,s},n-p-k_1}(y) \\
+y \binom{n-p-2}{k_1-1} \sum_{a=1}^{k_1-1} U_{\Gamma_{k_1,k_1,s},n-p-k_1-a}(y)\Bigg).
\end{multline*}
\end{theorem}

For example, if $k_1 =2$, then 
\begin{align*}
U_{\Gamma_{2,2,s},n}(y) &= -y U_{\Gamma_{2,2,s},n-1}(y) \nonumber\\
& -y\left(\sum_{p=0}^{s-2} (n-p-1)  U_{\Gamma_{2,2,s},n-2-p}(y)+(n-p-2)yU_{\Gamma_{2,2,s},n-3-p}(y)\right).
\end{align*}

We shall further explore two special cases, namely,  
$k_1 =k_2 =s =2$ where the recursion 
becomes 
\begin{equation}\label{222rec}
U_{\Gamma_{2,2,2},n}(y) = -y U_{\Gamma_{2,2,2},n-1}(y) -
y(n-1)U_{\Gamma_{2,2,2},n-2}(y)  -y^2 (n-2)U_{\Gamma_{2,2,2},n-3}(y)
\end{equation} 
for $n > 1,$ and $k_1 =k_2 =2,s=3$ where the recursion 
becomes 
\begin{align}\label{223rec}
U_{\Gamma_{2,2,3},n}(y) = &-y U_{\Gamma_{2,2,3},n-1}(y) -
y(n-1)U_{\Gamma_{2,2,3},n-2}(y)  -y^2 (n-2)U_{\Gamma_{2,2,3},n-3}(y)- 
\nonumber \\
&y(n-2)U_{\Gamma_{2,2,3},n-3}(y)  -y^2 (n-3)U_{\Gamma_{2,2,3},n-4}(y).
\end{align} 

Tables \ref{tab:U1324_1423_123even} and \ref{tab:U1324_1423_123odd} below give the polynomials $U_{\Gamma_{2,2,2}, n}(-y)$ for even and odd values of $n,$ respectively.
 
\begin{table}[h]
\begin{center}
\begin{tabular}{c|c|l}
$k$ & $n$ & $U_{\Gamma_{2,2,2},2k}(-y)$ \\ 
\hline 1 & 2 & $y + y^2$\\
2 & 4 & $3y^2 + 3y^3 + y^4$ \\
3 & 6 & $15y^3 + 15y^4 + 5y^5 + y^6$\\
4 & 8 & $105y^4 + 105y^5 + 35y^6 + 7y^7 + y^8$\\
5 & 10 & $945y^5 + 945y^6 + 315y^7 + 63y^8 + 9y^9 + y^{10}$\\
6 & 12 & $10395y^6 + 10395y^7 + 3465y^8 + 693y^9 + 99y^{10}+ 11y^{11}+ y^{12}$ \\ 
7 & 14 & $ 135135y^7 + 135135y^8 + 45045y^9 +  9009y^{10} + 1287y^{11} +143y^{12} + 13y^{13} + y^{14}$ \\
\end{tabular} 
\caption{The polynomials $U_{\Gamma_{2,2,2},2k}(-y)$ for $\Gamma_{2,2,2} = \{1324,1423,123\}$} \label{tab:U1324_1423_123even}
\end{center}
\end{table}

\begin{table}[h]
\begin{center}
\begin{tabular}{c|c|l}
$k$ & $n$ & $U_{\Gamma_{2,2,2},2k+1}(-y)$ \\ 
\hline 1 & 3 & $2y^2+y^3$\\
2 & 5 & $8y^3 + 4y^4 + y^5$ \\
3 & 7 & $48y^4 + 24y^5 + 6y^6 + y^7$\\
4 & 9 & $384y^5 + 192y^6 + 48y^7 + 8y^8 + y^9$\\
5 & 11 & $3840y^6 + 1920y^7 + 480y^8 + 80y^9 + 10y^{10} + y^{11}$\\
6 & 13 & $46080y^7 + 23040^8 + 5760^9 + 960y^{10} + 120y^{11}+ 12y^{12}+ y^{13}$ \\ 
7 & 15 & $645120y^8 + 322560y^9 + 80640y^{10} +  13440y^{11} + 1680y^{12} +168y^{13} + 14y^{14} + y^{15}$ \\
\end{tabular} 
\caption{The polynomials $U_{\Gamma_{2,2,2},2k+1}(-y)$ for $\Gamma_{2,2,2} = \{1324,1423,123\}$} \label{tab:U1324_1423_123odd}
\end{center}
\end{table}
 
This data leads us to conjecture the following explicit 
formulas: 
\begin{align} \displaystyle
U_{\Gamma_{2,2,2}, 2k}(-y) = \sum_{i=0}^{k} (2k-1)\downarrow\downarrow_{k-i}y^{k+i} \label{eq:U_even} \\
U_{\Gamma_{2,2,2}, 2k+1}(-y) = \sum_{i=0}^{k} (2k)\downarrow\downarrow_{k-i}y^{k+1+i} \label{eq:U_odd}
\end{align} where $(x)\downarrow \downarrow_{0} =1$ and 
 $(x)\downarrow \downarrow_{k} =x(x-2)(x-4) \cdots (x-2k -2)$ for $k \geq 1$. 

These formulas can be proved by induction. Note that it follows 
from (\ref{222rec}) that for $n > 1$, 
\begin{align}\label{eq:U1324_1423_123}
U_{\Gamma_{2,2,2},n}(-y) = yU_{\Gamma_{2,2,2},n-1}(-y)+ y(n-1)U_{\Gamma_{2,2,2},n-2}(-y) -  y^2(n-2)U_{\Gamma_{2,2,2},n-3}(-y). 
\end{align}

One can directly check these formulas for $n \leq 3$. For $n > 3$, 
let $U_{\Gamma_{2,2,2},n}(-y)|_{y^{k}}$ be the coefficient of $y^{k}$ in $U_{\Gamma_{2,2,2},n}(-y).$ Equation (\ref{eq:U1324_1423_123}) allows us to write the coefficient of $y^{k+1+i},$ for $0 \leq i \leq k,$ in $U_{\Gamma_{2,2,2}, 2k+1}(-y)$ as  
\begin{align*}
U_{\Gamma_{2,2,2},2k+1}(-y)|_{y^{k+1+i}} &=  U_{\Gamma_{2,2,2},2k}(-y)|_{y^{k+i}} +  (2k)U_{\Gamma_{2,2,2},2k-1}(-y)|_{y^{k+i}} \\
&\quad - (2k-1)  U_{\Gamma_{2,2,2},2k-2}(-y)|_{y^{k+i-1}}\\
&= (2k-1)\downarrow\downarrow_{k-i} + (2k)(2k-2)\downarrow\downarrow_{k-i} - (2k-1)\cdot(2k-3)\downarrow\downarrow_{k-i} \\
& = (2k)\downarrow\downarrow_{k-i}.
\end{align*}

For the even case when $n=2k,$ the coefficient of $y^{k+i},$ for $0 \leq i \leq k,$ in $U_{\Gamma_{2,2,2}, 2k}(-y)$ is   
\begin{align*}
U_{\Gamma_{2,2,2},2k}(-y)|_{y^{k+i}} &=  U_{\Gamma_{2,2,2},2k-1}(-y)|_{y^{k+i-1}} +  (2k-1)U_{\Gamma_{2,2,2},2k-2}(-y)|_{y^{k+i-1}} \\
&\quad - (2k-2)  U_{\Gamma_{2,2,2},2k-3}(-y)|_{y^{k+i-2}}\\
&= (2k-2)\downarrow\downarrow_{k-i} + (2k-1)(2k-3)\downarrow\downarrow_{k-i} - (2k-2)\cdot(2k-4)\downarrow\downarrow_{k-i} \\
& = (2k-1)\downarrow\downarrow_{k-i}.
\end{align*}
This proves equations (\ref{eq:U_even}) and (\ref{eq:U_odd}).

Hence, we can give a closed formula for 
$\mbox{NM}_{\Gamma_{2,2,2}}(t,x,y)$. That is, we have the following 
theorem.
\begin{theorem}
\begin{multline*}
\mbox{NM}_{\Gamma_{2,2,2}}(t,x,y) = \\
\left(\frac{1}{1+ \left(\sum_{n \geq 1}\frac{t^n}{n!} \sum_{i=0}^{k} (2k-1)\downarrow\downarrow_{k-i}y^{k+i}\right) + \left(\sum_{n \geq 0} \frac{t^n}{n!} \sum_{i=0}^{k} (2k)\downarrow\downarrow_{k-i}y^{k+1+i}\right)}\right)^x.
\end{multline*}
\end{theorem}

It follows from (\ref{223rec}) that

\begin{align*}
U_{\Gamma_{2,2,3},n}(-y) & = yU_{\Gamma_{2,2,3},n-1}(-y) + y(n-1)U_{\Gamma_{2,2,3},n-2}(-y) + y(n-2)U_{\Gamma,n-3}(-y) \nonumber \\ 
& \qquad \qquad  -y^2(n-2)U_{\Gamma_{2,2,3},n-3}(-y) -
y^2(n-3)U_{\Gamma_{2,2,3},n-4}(-y). 
\end{align*}

The next three tables below give the polynomials $U_{\Gamma_{2,2,3}, n}(y)$ for $n = 3k,n=3k+1,$ and $n=3k+2,$ respectively.

\begin{table}[h]
\begin{center}
\begin{tabular}{c|c|l}
$k$ & $n$ & $U_{\Gamma_{2,2,3},3k}(-y)$ \\ 
\hline 1 & 3 & $y + 2y^2 + y^3$\\
2 & 6 & $4y^2 + 33y^3 + 19y^4 + 5y^5 + y^6$ \\
3 & 9 & $28y^3 + 767y^4 + 781y^5 + 267y^6 + 55y^7 + 8y^8 + y^9$\\
4 & 12 & $280y^4 + 20496y^5 + 44341y^6 + 20765y^7 + 5137y^8 + 861y^9 $\\
 & & \hspace{200pt} $+ 109y^{10} + 11y^{11} + y^{12}$\\
5 & 15 & $3640y^5 + 598892y^6 + 2825491y^7 + 2072739y^8 + 641551y^9 + 125111y^{10} $\\
& & \hspace{160pt} $ + 17755y^{11} + 1977y^{12} 181y^{13} + 14y^{14} + y^{15}  $
\end{tabular} 
\caption{The polynomials $U_{\Gamma_{2,2,3},3k}(-y)$ for $\Gamma_{2,2,3} = \{1324,1423,1234\}$} \label{tab:U1234_3k}
\end{center}
\end{table}

\begin{table}[h]
\begin{center}
\begin{tabular}{c|c|l}
$k$ & $n$ & $U_{\Gamma_{2,2,3},3k+1}(-y)$ \\ 
\hline 1 & 4 & $5y^2 + 3y^3 + y^4$\\
2 & 7 & $67y^3 + 81y^4 + 29y^5 + 6y^6 + y^7 $ \\
3 & 10 & $1166y^4 + 3321y^5 + 1645y^6 +417y^7 + 71y^8 + 9y^9 + y^{10}$\\
4 & 13 & $23746y^5 + 160647y^6 + 128771y^7 + 41055y^8 + 8137y^9 + 1167y^{10}$\\
& & \hspace{220pt} $ + 131y^{11} + 12y^{12} + y^{13}$\\
5 & 16 & $550844y^6 + 8107518y^7 + 12109429y^8 + 5170965y^9 + 1225973y^{10}  $\\
& & \hspace{100pt} $+ 200253y^{11} + 24889y^{12} + 2493y^{13}+ 209y^{14} + 15y^{15} + y^{16}  $
\end{tabular} 
\caption{The polynomials $U_{\Gamma_{2,2,3},3k+1}(-y)$ for $\Gamma_{2,2,3} = \{1324,1423,1234\}$} \label{tab:U1234_3k+1}
\end{center}
\end{table}

\begin{table}[h]
\begin{center}
\begin{tabular}{c|c|l}
$k$ & $n$ & $U_{\Gamma_{2,2,3},3k+2}(-y)$ \\ 
\hline 1 & 5 & $7y^2 + 11y^3 + 4y^4 + y^5 $ \\
2 & 8 & $70y^3 + 297y^4 + 157y^5 + 41y^6 + 7y^7 + y^8$\\
3 & 11 & $910y^4 + 10343y^5 + 9223y^6 + 3069y^7 + 613y^8 + 89y^9 + 10y^{10} + y^{11} $\\
4 & 14 & $14560y^5 + 390564y^6 + 687109y^7 + 306413y^8 + 74137y^9 + 12261y^{10} + 1537y^{11}$\\
& & \qquad \qquad \qquad \qquad   $+ 155y^{12} + 13y^{13} +  y^{14}$
\end{tabular} 
\caption{The polynomials $U_{\Gamma_{2,2,3},3k+2}(-y)$ for $\Gamma_{2,2,3} = \{1324,1423,1234\}$} \label{tab:U1234_3k+2}
\end{center}
\end{table}

For any $s \geq 3$, it is easy to see that the lowest power of $y$ that 
occurs in $U_{\Gamma_{2,2,s},n}(-y)$ corresponds to 
brick tabloids where we use the minimum number of 
bricks. Since the maximum size of brick in 
a fixed point of $I_{\Gamma_{2,2,s}}$ is $s$, we see 
that the minimum number of bricks that we can 
use for a fixed point of  $I_{\Gamma_{2,2,s}}$ of 
length $sn$ is $n$ while the minimum number of bricks that we can 
use for a fixed point of  $I_{\Gamma_{2,2,s}}$ of 
length $sn+j$ for $1 \leq j \leq s-1$  is $n+1$. 
We can prove the following general theorem for 
the coefficients of the lowest power of $y$ that 
appears in $U_{\Gamma_{2,2,s},n}(-y)$.  
\begin{theorem}\label{thm:sn}
For $n \geq 1$,
\begin{equation}\label{eq:sn}
U_{\Gamma_{2,2,s},sn}(-y)|_{y^n} = \prod_{i=1}^n ((i-1)s+1)
\end{equation}
and 
\begin{equation}\label{eq:sn+s-1}
U_{\Gamma_{2,2,s},sn+s-1}(-y)|_{y^{n+1}} = \prod_{i=1}^n ((i+1)s+1).
\end{equation}
\end{theorem}
\begin{proof}

For (\ref{eq:sn}), we first notice that any fixed point $(B,\sg)$ of $I_{\Gamma_{2,2,s}}$ that contributes to $U_{\Gamma_{2,2,s},sn}(-y)|_{y^n}$ must 
have only bricks of size $s$.  Thus $B=(s,\ldots, s)$. We shall 
prove (\ref{eq:sn}) by induction on $n$. Clearly, 
$U_{\Gamma_{2,2,s},s}(-y)|_{y} =1$.  Now suppose $(B,\sg)$ 
is a fixed point of $I_{\Gamma_{2,2,s}}$ of size $sn$ where 
$\sg = \sg_1 \cdots \sg_{sn}$. By 
our arguments above, the first $s-1$ elements 
of the first brick must be $1,2, \ldots, s-1$, reading from 
left to right. The element in the next cell $\sg_s$ can be arbitrary. 
That is, if it is equal to $s$, then there will be an increase 
between the first two bricks and if $\sg_s >s$, then it 
must be the case that $\sg_{s+1} =s$ in which case there 
will by $\Gamma_{2,2,s}$-match that involves the last two 
cells of the first brick and the first two cells of the next brick. 
We can then remove the first brick and adjust the remaining numbers 
to produce a fixed point $O'$ of $I_{\Gamma_{2,2,s}}$ of 
length $s(n-1)$ in which every brick is of size $s$. It 
follows by induction that 
\begin{align*} 
U_{\Gamma_{2,2,s},sn}(-y)|_{y^n} & =((n-1)s+1)
U_{\Gamma_{2,2,s},s(n-1)}(-y)|_{y^{n-1}} \\
&= ((n-1)s+1) \prod_{i=1}^{n-1} \left((i-1)s+1\right) \\
&= \prod_{i=1}^n \left((i-1)s+1\right).
\end{align*}

Next consider $U_{\Gamma_{2,2,s},2s-1}(-y)|_{y^2}$. 
In this case, either the first brick of size $s-1$ or 
the first brick is of size $s$. If the first brick 
is of size $s$, then we can argue as above that 
the first $s-1$ elements of the first brick are $1, \ldots, s-1$, 
and we have $s$ choices for the last element of the first brick. 
If the first brick 
is of size $s-1$, then we can argue as above that 
the first $s-2$ elements of the first brick are $1, \ldots, s-2$, 
and we have $s+1$ choices for the last element of the first brick. 
Thus 
$$U_{\Gamma_{2,2,s},2s-1}(-y)|_{y^2} = 2s+1.$$

Next consider $U_{\Gamma_{2,2,s},(ns+s-1)}(-y)|_{y^{n+1}}$. 
In such a situation, any fixed point $(B,\sg)$ of 
$I_{\Gamma_{2,2,s}}$ that can contribute to 
$U_{\Gamma_{2,2,s},(ns+s-1)}(-y)|_{y^{n+1}}$ must have 
$n$ bricks of size $s$ and one brick of size $s-1$. 
If the first brick 
is of size $s$, then we can argue as above that 
the first $s-1$ elements of the first brick are $1, \ldots, s-1$, 
and we have $sn$ choices for the last element of the first brick. 
Then we can remove this first brick and adjust 
the remaining numbers to produce a fixed point 
$O'$ in $I_{\Gamma_{2,2,s}}$ of size $(n-1)s+s-1$ which 
has $n-1$ bricks of size $s$ and one brick of size $s-1$.
If the first brick 
is of size $s-1$, then we can argue as above that 
the first $s-2$ elements of the first brick are $1, \ldots, s-2$, 
and we have $sn+1$ choices for the last element of the first brick.
Then we can remove this first brick and adjust 
the remaining numbers to produce a fixed point 
$O'$ in $I_{\Gamma_{2,2,s}}$ of size $ns$ which 
has $n$ bricks of size $s$ 

Thus if $n \geq 2$, 
\begin{align*}
U_{\Gamma_{2,2,s},(ns+s-1)}(-y)|_{y^{n+1}} & = 
(sn+1) U_{\Gamma_{2,2,s},ns}(-y)|_{y^{n}} + 
(sn) U_{\Gamma_{2,2,s},((n-1)s+s-1)}(-y)|_{y^{n}} \\
& = (sn+1)\prod_{i=1}^n ((i-1)s+1) + 
(sn) \prod_{i=1}^{n-1}((i+1)s+1) \\
&= (s+1) \prod_{i=1}^{n-1} ((i+1)s+1) + (sn)\prod_{i=1}^{n-1}((i+1)s+1) \\
&= ((n+1)s+1) \prod_{i=1}^{n-1}((i+1)s+1) \\
&= \prod_{i=1}^{n}((i+1)s+1).
\end{align*}
\end{proof}

Unfortunately, we cannot extend this type of argument 
to compute $U_{\Gamma_{2,2,s},ns+k}(-y)|_{y^{n+1}}$ where 
$1 \leq k \leq s-2$.  The problem is that we have more than  
one choice for the sizes of the bricks in such cases. For example, to 
compute  $U_{\Gamma_{2,2,3},4}(-y)|_{y^{3}}$, 
the brick sizes could be some rearrangement of 
(3,1) or (2,2). One can use our recursions to 
compute $U_{\Gamma_{2,2,s},ns+k}(-y)|_{y^{n+1}}$ for 
small values of $s$. For example, we can find all 
the coefficients of the lowest power 
of $U_{\Gamma_{2,2,3},n}(-y)$.   That is, we claim 
 \begin{itemize}
\item[(i)] $ U_{\Gamma_{2,2,3},3k}(-y)|_{y^k} = \prod_{i = 1}^{k}(3(i-1)+1)$,
\item[(ii)] $U_{\Gamma_{2,2,3},3k+2}(-y)|_{y^{k+1}} = \prod_{i=1}^{k}(3(i+1)+1)$, and 
\item[(iii)] if $A_k = U_{\Gamma,3k+1}(-y)|_{y^{k+1}} $ then $A_1 = 5$ and $A_{k} = (3k-1)A_{k-1} + (3k)\prod_{i=1}^{k-1}(3i+4)$ for all $k \geq 2$.
\end{itemize} 
Clearly, (i) and (ii) follow from our previous theorem. 
To prove (iii), note that \begin{align*}
A_k= U_{\Gamma,3k+1}(-y)|_{y^{k+1}}& = U_{\Gamma,3k}(-y)|_{y^{k}} + (3k) U_{\Gamma,3k-1}(-y)|_{y^{k}} + (3k-1) U_{\Gamma,3k-2}(-y)|_{y^{k}} \nonumber \\
& \hspace{50 pt} - (3k-1) U_{\Gamma,3k-2}(-y)|_{y^{k-1}} - (3k-2) U_{\Gamma,3k-3}(-y)|_{y^{k-1}} \nonumber \\
& = \prod_{i=1}^{k}(3i-2) +(3k)\prod_{i=1}^{k-1}(3i+4) + (3k-1)U_{\Gamma,3k-2}(-y)|_{y^{k}} \nonumber \\
& \hspace{100pt}  - (3k-2)\prod_{i=1}^{k-1}(3i-2)\nonumber \\
& = (3k)\prod_{i=1}^{k-1}(3i+4) + (3k-1)U_{\Gamma,3k-2}(-y)|_{y^{k}}\\
&= (3k-1)A_{k-1} + (3k)\prod_{i=1}^{k-1}(3i+4).
\end{align*} 
This explains all the coefficients for the smallest power of $y$ in the polynomials $U_{\Gamma_{2,2,3},n}(-y)$ for the family $\Gamma_{2,2,3} = \{ 1324,1423,1234  \}.$

\section{The Proofs of Theorem \ref{thm:1324123} and Theorem \ref{thm:1324p}} 
In this section, we will study two more examples of 
the differences between the recursions for 
$U_{\Gamma,n}(y)$'s and the recursions 
for $U_{\Gamma \cup \{12 \cdots s(s+1)\},n}(y)$'s. In 
particular, we will prove Theorems  \ref{thm:1324123} and \ref{thm:1324p}.\\
\ \\
{\bf Proof of Theorem \ref{thm:1324123}}\\
\ \\
Let $\Gamma = \{ 1324, 123\}$.  Let $(B,\sg)$ be a fixed 
point $I_{\Gamma}$ where $B=(b_1, \ldots, b_k)$ and 
$\sg = \sg_1 \cdots \sg_n$. By Lemma \ref{lem:keyGamma}, 
we know that all the bricks $b_i$ 
must be of size 1 or 2. Since the minimal elements 
in bricks of $B$ must weakly increase, we 
see that 1 must be in cell 1 and 2 must be either 
in $b_1$ or it is in the first cell of $b_2$.  Thus we 
have three possibilities.\\ 
\ \\
{\bf Case 1.} 2 is in $b_1$.\\
\ \\
In this case, $b_1$ must be of size 2 and we can remove 
$b_1$ from $(B,\sg)$ are reduce the remaining numbers 
by 2 to get a fixed point of $I_\Gamma$ of size 
$n-2$.  It then easily follows that the 
fixed points in Case 1 contribute 
$-yU_{\Gamma,n-2}(y)$ to $U_{\Gamma,n}(y)$. \\
\ \\
{\bf Case 2.} 2 is in $b_2$ and $b_1 =1$.\\
\ \\
In this case, it is easy to see that 1 cannot be involved 
in any $\Gamma$-match so that 
we can remove 
$b_1$ from $(B,\sg)$ are reduce the remaining numbers 
by 1 to get a fixed point of $I_\Gamma$ of size 
$n-1$.  It follows that the  
fixed points in Case 2 contribute 
$-yU_{\Gamma,n-1}(y)$ to $U_{\Gamma,n}(y)$.\\
\ \\
{\bf Case 3.} 2 is in $b_2$ and $b_1 =2$.\\
\ \\
In this case, there is descent between bricks $b_1$ and $b_2$ so 
that there must be a 1324-match in $\sg$ contained 
in the cells of $b_1$ and $b_2$. In particular, 
this means $b_2 =2$ and there is 1324-match starting at 1 in 
$\sg$. 
We then have two subcases. \\
\ \\
{\bf Subcase 3.a.} There is no $1324$-match in $(B,\sg)$ starting at cell 3 \\
\ \\
We claim that $ \{ \sg_1,\ldots,\sg_4  \}  = \{1,2,3,4\}$. 
If not, let $d = \min(\{1,2,3,4\}-\{\sg_1, \ldots, \sg_4\}).$ Then $d$ must be in cell 5, the first cell of brick $b_3$ and there is a decrease between bricks $b_2$ and $b_3$ since $d \leq 4 < \sg_4.$ Thus, in order to avoid combining bricks $b_2$ and $b_3,$ we need a $1324$-match among the cells of these two bricks. However, the only possible $1324$-match among the cells of $b_2$ and $b_3$ would have to start at cell $3$ where $\sg_3=2.$ This contradicts the assumption that there is no $1324$-match in $(B,\sg)$ starting at cell 3. As a result, it must be the case that the first four numbers must occupy the first four cells of 
$(B,\sg)$ so we must have $\sg_1 =1$, $\sg_2=3$, 
$\sg_3 =2$, $\sg_4 =4$, and $\sg_5 =5$.
It then follows that if we let $O'$ be the result by removing the first four cells from $(B,\sg)$ and then subtract 4 from the remaining entries in the cells, then $O'$ will be a fixed point in $\mathcal{O}_{\Gamma,n-4}$. It 
then easily follows that the contribution of fixed points in subcase 3.a to $U_{\Gamma,n}(y)$ is $(-y)^2U_{\Gamma,n-4}(y)$.

\ \\
{\bf Subcase 3.b.} There is a $1324$-match in $O$ starting at cell $3$ \\
\ \\
In this case, there is decrease between bricks $b_2$ and $b_3$. Hence, the 1324-match starting at cell 3 must be 
contained in the cells of $b_2$ and $b_3$ so that $b_3$ must be of size 2. 
In general, suppose that the bricks $b_2, \dots, b_{k-1}$ all have exactly two cells and there are $1324$-matches starting at cells $1,3,\dots, 2k-3$ but there is no $1324$-match starting at cell $2k-1$ in $O.$ 

Similar to Subcase 3.a, we will show that $ \{ \sg_1, \ldots, \sg_{2k}  \}  = \{1,2,\dots,2k\}$. That is, 
the first $2k$ numbers must occupy the first $2k$ cells in $O$. If not, 
let $d = \min(\{1,2,\dots,2k\}-\{\sg_1, \ldots, \sg_{2k}\})$. Since 
the minimal elements of the bricks are weakly increasing, it 
must be the case that $d$ is in the first cell of $b_{k+1}$. 
Next, the fact that there are 1324-matches starting 
in cells $1,3, \ldots, 2k-1$ easily implies that 
$\sg_{2k}$ is the largest element in $\{\sg_1, \ldots, \sg_{2k}\}$ 
which means that $\sg_{2k}>d$.  But then there is a 
decrease between bricks $b_k$ and $b_{k+1}$ which means 
that there must be a 1324-match contained in the cells of 
$b_k$ and $b_{k+1}$.  This implies that there 
is a 1324-match starting at cell $2k-1$ which contradicts our 
assumption.

Thus, if we remove the first $2k$ cells of $(B,\sg)$ 
and subtract $2k$ from the remaining elements, 
we will obtain a fixed point $O'$ in $\mathcal{O}_{\Gamma, n - 2k}.$ Therefore, each fixed point $O$ in this case will contribute $(-y)^{k}U_{\Gamma, n-2k}(y)$ to $U_{\Gamma, n}(y).$ The final task is to count the number of 
permutations $\sg_1 \cdots \sg_{2k}$ of $\mathfrak{S}_{2k}$ that 
has 1324-matches starting at positions $1,3, \ldots, 2k-3$. In \cite{JR2}, Jones and Remmel gave  a bijection between the set of such 
$\sg$  and the set of Dyck paths of length $2k-2$. 
Hence, there are $C_{k-1}$ such fixed points, where $C_{n} = \frac{1}{n-1}\binom{2n}{n}$ is the $n^{th}$ Catalan number. It then easily follows that 
the contribution of the fixed points in Subcase 3.b to $U_{\Gamma,n}(y)$ is 
\[ \sum_{k=2}^{\Floor }(-y)^{k}C_{k-1}U_{\Gamma, n-2k}(y).  \]

Hence, we  know that $U_{\Gamma,1}=-y$ and for $n > 1$, 
\begin{align*}
\displaystyle  U_{\Gamma,n}(y) &= -yU_{\Gamma,n-1}(y) -yU_{\Gamma,n-2}(y)
+ \sum_{k=2}^{\Floor }(-y)^{k}C_{k-1}U_{\Gamma, n-2k}(y). 
\end{align*} 
This proves Theorem  \ref{thm:1324123}.

We have  computed the polynomials $U_{\{1324,123\},n}(-y)$ for small $n$ 
which are given in the Table \ref{tab:U1324123} below. 
\begin{table}[h]
\begin{center}
\begin{tabular}{c|l}
n & $U_{\{1324,123\},n}(-y)$ \\ 
\hline 1 & $y$\\
2 & $y + y^2$ \\
3 & $2y^2 + y^3$\\
4 & $2y^2 +3y^3 + y^4$\\
5 & $5y^3 + 4y^4 + y^5$\\
6 & $5y^3 + 9y^4 +5y^5 + y^6$ \\ 
7 & $14y^4 +14y^5 + 6y^6 + y^7$ \\
8 & $14y^4 +28y^5 + 20y^6 + 7y^7 + y^8$ \\ 
9 & $42y^5 + 48y^6 + 27y^7 + 8y^8 + y^9 $  \\ 
10 & $42y^5  + 90y^6 + 75y^7 + 35y^8 + 9y^9 +y^{10}$  \\ 
\end{tabular} 
\end{center}
\caption{The polynomials $U_{\Gamma,n}(-y)$ for $\Gamma = \{1324,123\}$} \label{tab:U1324123}
\end{table}

An anonymous referee observed that up to a power of $y$, 
the odd rows are the triangle A039598 in the OEIS and 
the even rows are the triangle A039599 in the OEIS.  These 
tables arise in expanding the powers of $x$ in terms of 
the Chebyshev polynomials of the second kind. 
Since there are explicit formula for entries in these tables, 
we have the following theorem.

\begin{theorem} \label{1324,123}

Let $\Gamma = \{1324,123\}$. Then for all 
$n \geq 0$, 
\begin{equation}\label{u2n}
U_{\Gamma,2n}(y) = 
\sum_{k=0}^n \frac{(2k+1)\binom{2n}{n-k}}{n+k+1}(-y)^{n+k+1}
\end{equation}
and 
\begin{equation}\label{u2n+1}
U_{\Gamma,2n+1}(y) = 
\sum_{k=0}^n \frac{2(k+1)\binom{2n+1}{n-k}}{n+k+2}(-y)^{n+k}
\end{equation}
\end{theorem}
\begin{proof}

First we consider the polynomials $U_{\Gamma,2n+1}(-y)$ which 
correspond to the entries in the table $T(n,k)$ for $0 \leq k \leq n$ 
of entry A039598 in the OEIS. $T(n,k)$ has an explicit 
formula, namely,  
$$T(n,k) = \frac{2(k+1)\binom{2n+1}{n-k}}{n+k+2}$$ 
for all $n \geq 0$ and $0 \leq k \leq n$. Let $\mathcal{T}(n,k)$ be set 
all of paths of length $2n+1$ consisting 
of either up steps $(1,1)$ or down steps $(1,-1)$ that start 
at (0,0) and end at $(2n+1,2k+1)$ which stay above the 
$x$-axis. Then one of the combinatorial 
interpretations of the $T(n,k)$'s is that $T(n,k) =|\mathcal{T}(n,k)|$. Let 
$\mathcal{F}_{2n+1,2k+1}$ be the set of all 
fixed points of $I_{\Gamma}$ with $2k+1$ bricks of 
size 1 and $n-k$ bricks of size 2.  We will construct a bijection 
$\theta_{n,k}$ from $\mathcal{F}_{2n+1,2k+1}$ onto $\mathcal{T}(n,k)$.  
Note all 
$(B,\sg) \in  \mathcal{F}_{2n+1,2k+1}$ have weight $(-y)^{n+k+1}$ so that 
the bijections $\theta_{n,k}$ will prove (\ref{u2n+1}).

First we must examine the fixed points of $I_{\Gamma}$ in greater detail. 
Note that since $\Gamma$ contains the identity permutation 
$123$, all the bricks in any fixed point of $I_{\Gamma}$ must be of size 1 
or size 2. Next, we consider the structure of the fixed points 
of $I_{\Gamma}$ which have $k$ bricks of size 1 and 
$\ell$ bricks of size 2. Suppose $(B,\sg)$ is such a fixed 
point where $B=(b_1, \ldots, b_{k +\ell})$ and 
that the bricks of size 1 in $B$ are 
$b_{i_1}, \ldots, b_{i_k}$ where 
$1 \leq i_1 < \cdots < i_k \leq k+\ell$. For any $s$, there cannot 
be a decrease between brick $b_{i_j-1}$ and brick $b_{i_j}$ in $B$ since otherwise we could combine bricks $b_{i_j-1}$ and $b_{i_j}$, which would violate 
our assumption that $(B,\sg)$ is a fixed point of $I_{\Gamma}$. 
Next we claim that if there are $s$ bricks of size 
$2$ that come before brick $b_{i_j}$ so that 
$b_{i_j}$ covers cell $2s+j$ in $(B,\sg)$, then 
$\sg_{2s+j} = 2s+j$ and $\{\sg_1, \ldots, \sg_{2s+j}\} = \{1, \ldots, 2s+j\}$.
To prove this claim, we proceed by induction.  For the base case, 
suppose that $b_{i_1}$ covers cell $2s+1$ so that $(B,\sg)$ starts 
out with $s$ bricks of size 2.  If $s=0$, there is nothing 
to prove.  Next suppose that $s =1$. Then we know that 
in all fixed points of 
$I_{\Gamma}$, 2 must be in cell 2 or cell 3. Since 
there is an increase between $b_1$ and $b_2$, it must be the 
case that 1 and 2 lie in $b_1$ and since the minimal elements 
in the brick form a weakly increasing sequence, it must be the case 
that $b_2$ is filled with 3. If $s \geq 2$, then for 
$1 \leq i < s$, either there is an increase between $b_i$ and 
$b_{i+1}$ in which case the elements in $b_i$ and $b_{i+1}$ must 
match the pattern $1234$, or there is a decrease between $b_i$ and 
$b_{i+1}$ in which case the four  elements must match the pattern 
1324. This means that if for each brick of size $2$, we place 
the second element of the brick on the top of the first element, 
then any two consecutive bricks will be one of the two forms 
pictured in Figure \ref{fig:stack}. 
Thus if we consider the $s \times 2$ array built 
from the first $s$ bricks of size 2, we will obtain a column 
strict tableaux with distinct entries of shape $(s,s)$. 
In particular, it must be the case that the 
largest element in the array is the element which appears 
at the top of the last column. That element corresponds 
to the second cell of brick $b_s$. Since there is an increase 
between brick $b_s$ and brick $b_{s+1}$ it must mean that the 
element in brick $b_{s+1}$ is larger than any of the elements that appear 
in bricks $b_1, \ldots, b_s$. Thus $\sg_i < \sg_{2s+1}$ for 
$i \leq 2s$.  Since the minimal 
elements in the bricks are increasing, it follows 
that $\sg_{2s+1}< \sg_j$ for all $j > 2s+1$ so that it 
must be the case that $\sg_{2s+1} = 2s+1$ and 
$\{\sg_1, \ldots, \sg_{2s+1}\} = \{1, \ldots, 2s+1\}$. 
Thus the base case of our induction holds. 

\fig{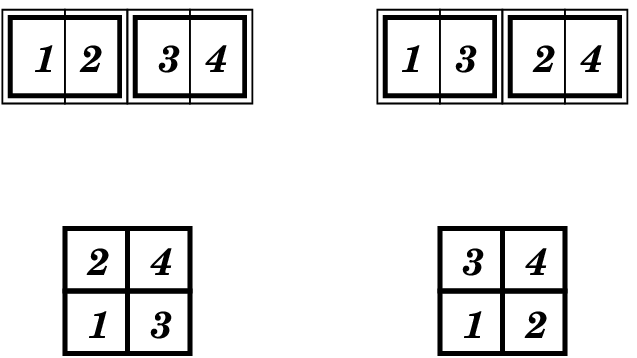}{Patterns for two consecutive brick of size 2 in 
a fixed point of $I_{\Gamma}$.}

We can repeat the same argument for $i_j$ where $j >1$. That 
is, by induction, we can assume that if 
there are $r$ bricks of size $2$ that precede brick $b_{i_{j-1}}$, 
then $\sg_{2r+j-1} = 2r+j-1$ and 
$\{\sg_1 , \ldots, \sg_{2r+j-1}\} = \{1, \ldots, 2r+j-1\}$. Hence if 
we remove these elements and subtract $2r+j-1$ from the remaining 
elements in $(B,\sg)$, we would 
end up with a fixed point of $I_{\gamma}$. Thus we can 
repeat our argument for the base case to prove that if 
there are $s$ bricks of size 2 between brick $b_{i_{j-1}}$ and 
$b_{i_j}$, then $\sg_{2r+2s+j} = 2r+2s+j$ and 
$\{\sg_1 , \ldots, \sg_{2r+2s +j}\} = \{1, \ldots, 2r+2s+j\}$.

Next we note that there is a well known bijection $\phi$ between 
standard tableaux of shape $(n,n)$ and Dyck paths of 
length $2n$, see \cite{Stanley}. Here a Dyck path is path consisting  
of either up steps $(1,1)$ or down steps $(1,-1)$ that starts  
at (0,0) and ends at $(2n,0)$ which stays above the 
$x$-axis. Given a standard tableau $T$, $\phi(T)$ is the Dyck path 
whose $i$-th segment is an up step if $i$ is the first row 
and whose $i$-th segment is a down step if $i$ is in the second row. 
This bijection is illustrated in Figure \ref{fig:catalanbij}. 

\fig{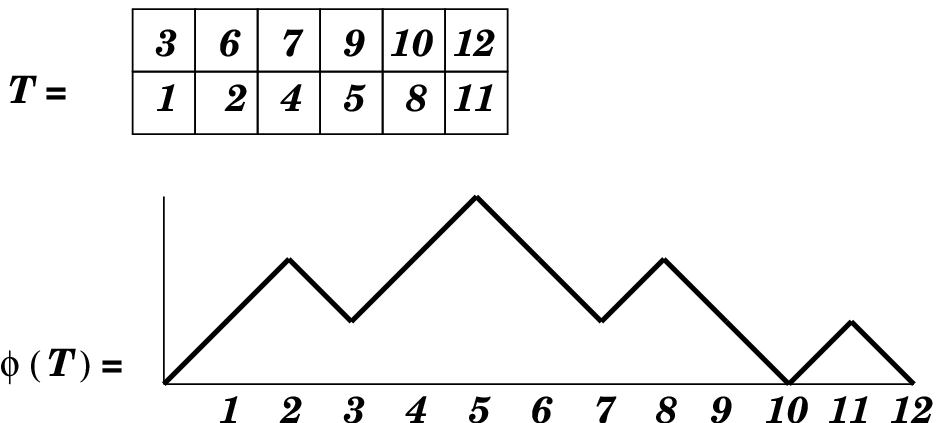}{The bijection $\phi$.}

We can now easily describe our desired bijection $\theta_{n,k}$. 
Starting with a fixed point $(B,\sg)$ in $\mathcal{F}_{2n+1,2k+1}$ 
where $B=(b_1, \ldots, b_{n+k+1})$, we 
can rotate all the bricks of size 2 by $-90$ degrees and 
end up with an array consisting of bricks of size one and 
$2 \times r$ arrays corresponding to standard tableaux. For example, 
this step is pictured in the second row of Figure \ref{fig:unoddbij}. 
By our remarks above, each $2 \times r$ array corresponds 
to standard tableaux of shape $(r,r)$ where the entries lie 
in some consecutive sequence of elements from $\{1, \ldots, 2n+1\}$. 
Suppose that $b_{i_1}, \ldots, b_{i_{2k+1}}$ are the bricks of 
size $1$ in $B$ where $i_1 < \cdots < i_{2k+1}$.  Let 
$T_j$ be the standard tableau corresponding to the consecutive 
string of brick of size 2 immediately preceding brick $b_{i_j}$ and 
$P_i$ be the Dyck path $\phi(T_i)$. If there is no bricks 
of size 2 immediately preceding $b_{i_j}$, then $P_j$ is just the 
empty path. Finally let $T_{2k+2}$ the standard tableau corresponding 
to the bricks of size 2 following $b_{i_{2k+1}}$ and $P_{2k+2}$ be 
the Dyck path corresponding to $\phi(T_{2k+2})$ where again 
$P_{2k+2}$ is the empty path if there are no bricks of size 2 following 
$b_{i_{2k+1}}$. Then 
$$\theta_{n,k}(B,\sg) = P_1(1,1)P_2(1,1) \ldots P_{2k+1}(1,1)P_{2k+2}.$$
For example, line 3 of Figure \ref{fig:unoddbij} illustrates this 
process. In fact, it easy to see that if $i$ is in the bottom 
row of intermediate diagram for $(B,\sg)$, then the $i$-th segment 
of $\theta_{n,k}(B,\sg)$ is an up step and if $i$ is in the top  
row of intermediate diagram for $(B,\sg)$, then the $i$-th segment 
of $\theta_{n,k}(B,\sg)$ is an down step.

\fig{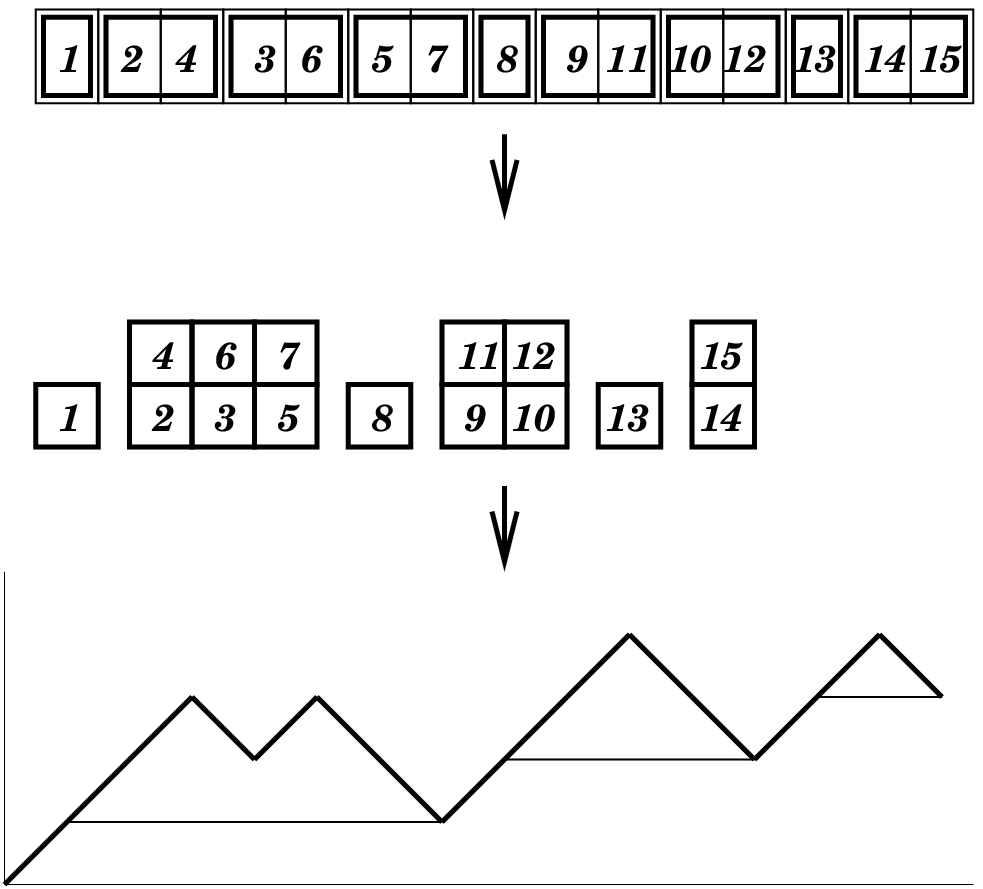}{The bijection $\theta_{n,k}$.}

The inverse of $\theta_{n,k}$ is also easy to describe. That is, 
given a path $P$ in $\mathcal{T}(n,k)$, we let 
$d_i$ be the step that corresponds to the last up step that ends at level $i$. 
Then $P$ can be factored as 
$$P_1 d_1 P_2 d_2 \ldots P_{2k+1}d_{2k+1}P_{2k+2}$$
where each $P_i$ is a path that corresponds to a Dyck path 
that starts at level $i-1$ and ends at level $i-1$ and stays above 
the line $x = i-1$. Then for each $i$, $T_i = \phi^{-1}(P_i)$ is a 
standard tableau. Using these tableaux and being cognizant of the 
restrictions on the initial segments of elements of $\mathcal{F}_{2n+1,2k+1}$ 
preceding bricks of size 1, one can easily reconstruct the 2 line 
intermediate array corresponding to 
$T_1 d_1 T_2 d_2 \ldots T_{2k+1} d_{2k+1} T_{2k+2}$. For example, 
this process is pictured on line 2 of Figure \ref{fig:unoddbij2}. 
Then we only have to rotate all the bricks of size corresponding to 
a bricks of height 2 by 90 degrees to obtain $\theta^{-1}_{n,k}(P)$. 
This step is pictured on line 3 of  Figure \ref{fig:unoddbij2}.

\fig{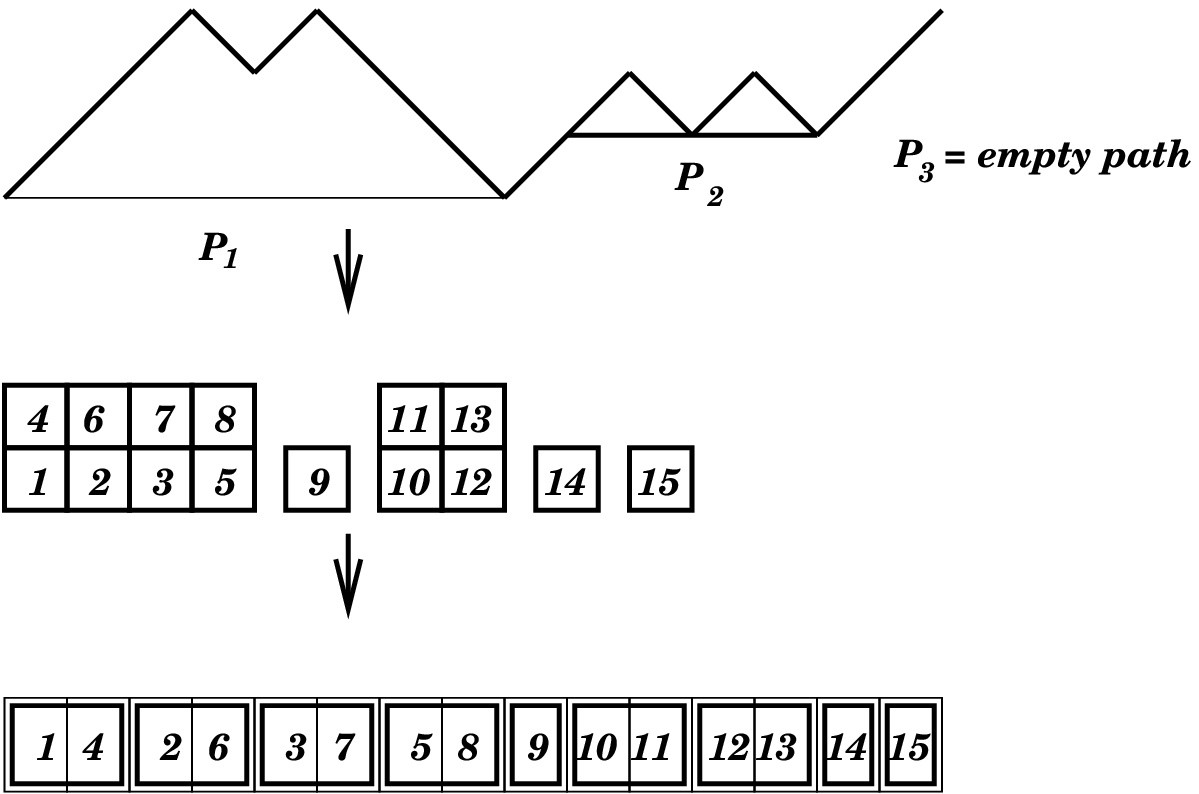}{The bijection $\theta^{-1}_{n,k}$.}

Next we consider the polynomials $U_{\Gamma,2n}(-y)$ which 
correspond to the entries in the table $R(n,k)$ for $0 \leq k \leq n$ 
of entry A039599 in the OEIS. $R(n,k)$ has an explicit 
formula, namely,  
$$R(n,k) = \frac{(2k+1)\binom{2n}{n-k}}{n+k+1}$$ 
for all $n \geq 0$ and $0 \leq k \leq n$. 
Let $\mathcal{R}(n,k)$ be set 
all of paths of length $2n$ consisting 
of either up steps $(1,1)$ or down steps $(1,-1)$ that start 
at (0,0) and end at $(2n,0)$ that have $k$ down steps that 
end on the line $x =0$. Here there is no requirement that the 
paths stay above the $x$-axis. Then one of the combinatorial 
interpretations of the $R(n,k)$s is that 
$R(n,k) =|\mathcal{R}(n,k)|$. Let 
$\mathcal{F}_{2n,2k}$ be the set of all 
fixed points of $I_{\Gamma}$ with $2k$ bricks of 
size 1 and $n-k$ bricks of size 2.  We will construct a bijection 
$\beta_{n,k}$ from $\mathcal{F}_{2n,2k}$ onto $\mathcal{R}(n,k)$.  Note all 
$(B,\sg) \in  \mathcal{F}_{2n,2k}$ weight $(-y)^{n+k}$ so that the bijections 
$\beta_{n,k}$ will prove (\ref{u2n+1}).

We can now easily describe our desired bijection $\beta_{n,k}$. 
Starting with a fixed point $(B,\sg)$ in $\mathcal{F}_{2n,2k1}$ 
where $B=(b_1, \ldots, b_{n+k})$, we 
can rotate all the bricks of size 2 by $-90$ degrees and 
end up with an array consisting of bricks of size one and 
$2 \times r$ arrays corresponding to standard tableaux. For example, 
this step is pictured in the second row of Figure \ref{fig:unevenbij}. 
By our remarks above, each $2 \times r$ array corresponds 
to standard tableaux of shape $(r,r)$ where the entries lie 
in some consecutive sequence of elements from $\{1, \ldots, 2n\}$. 
Suppose that $b_{i_1}, \ldots, b_{i_{2k}}$ are the bricks of 
size $1$ in $B$ where $i_1 < \cdots < i_{2k}$.  Let $T_s$ 
be the standard tableau corresponding to the bricks 
of size 2 immediately preceding brick $b_{j_s}$ for 
$1 \leq s \leq 2n$ and let $T_{2k+1}$ 
be the standard tableau corresponding to the bricks 
of size 2 following  brick $b_{i_{2k}}$. For $i =0, \ldots, 
2k+1$, let  $P_i$ be the Dyck path $\phi(T_i)$. In each 
case $j$ where there are no such bricks of size 2, then $P_j$ is just the 
empty path.  For each such $i$, let $\overline{P}_i$ denote 
the flip of $P_i$, i.e. the 
path that is obtained by flipping $P_i$ about the x-axis.  For example, 
the process of flipping a Dyck path is pictured in Figure \ref{fig:flip}. 

\fig{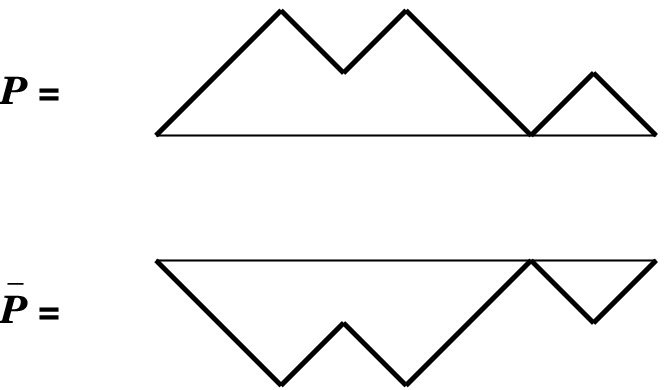}{The flip of Dyck path.}

Then 
$$\beta_{n,k}(B,\sg) = \overline{P}_1(1,1)P_2(1,-1) 
\overline{P}_3(1,1)P_4(1,-1) \ldots \overline{P}_{2k-1}(1,1)P_{2k}(1,-1)
\overline{P}_{2k+1}.$$
That is, each pair $b_{i_{2j-1}},b_{i_{2j}}$ will correspond 
to an up step starting at $x =0$ followed by a Dyck path 
which starts at ends a line $x=1$ followed by down step ending at $x=0$. 
These segments are then connected by flips of Dyck path that 
stay below the $x$-axis. Thus $\beta_{n,k}(B,\sg)$ will have 
exactly $k$ down steps that end at $x=0$. 
For example, line 3 of Figure \ref{fig:unevenbij} illustrates this 
process.

\fig{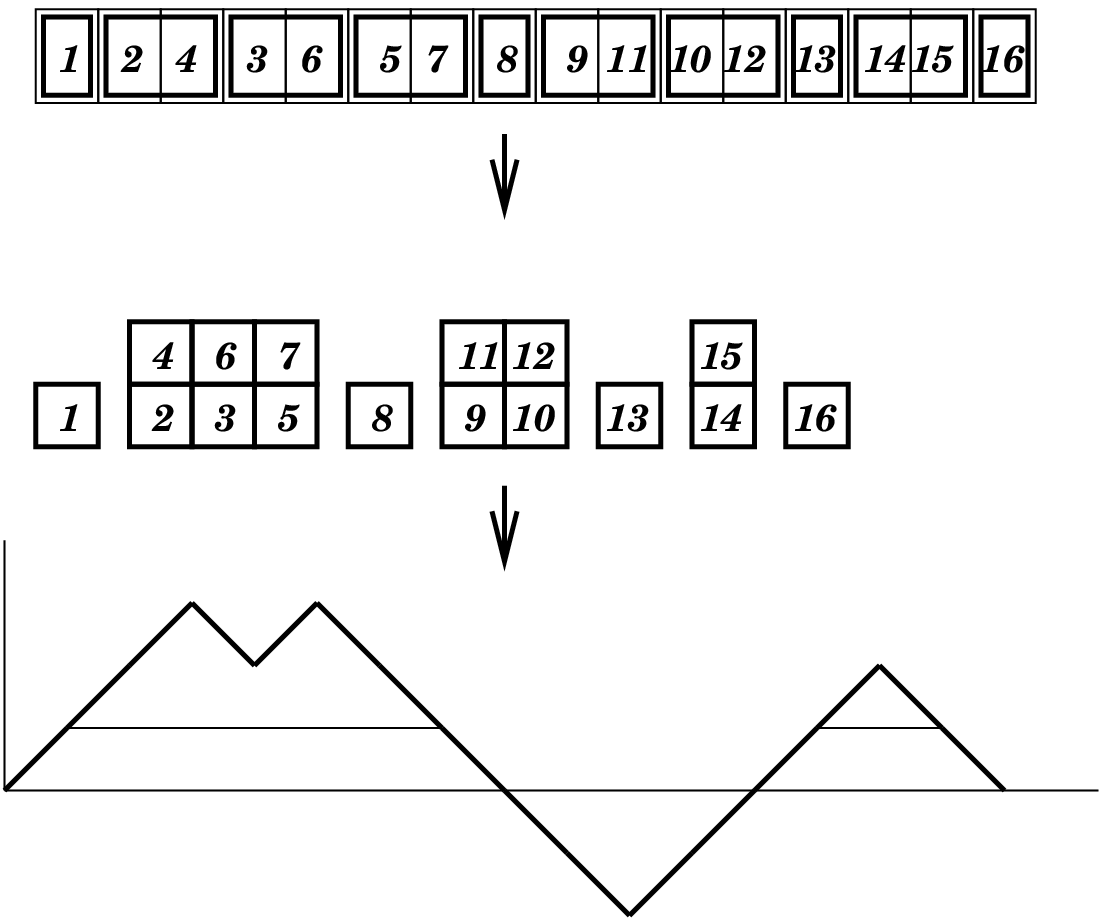}{The bijection $\beta_{n,k}$.}

The inverse of $\beta_{n,k}$ is also easy to describe. That is, 
given a path $P$ in $\mathcal{R}(n,k)$, 
let $f_1, \ldots, f_k$ be the positions of the down steps that 
end at $x=0$ and define $e_1, \ldots, e_k$ so that 
$e_1$ is the right most up step that starts at $x=0$ and precedes $f_1$ and 
for $2 \leq i \leq k$,, $e_i$ is the right most up step that 
follows $f_{i-1}$ and precedes $f_i$. It is then easy to 
see that the path $Q_1$ which precedes $e_1$ must be 
a path that starts at (0,0) and ends at $(e_1-1,0)$ and stays 
below the $x$-axis so that $Q_1$ is the flip of some Dyck path $P_1$. 
Next, the path $Q_2$ between $(e_1,1)$ and $(f_1-1,1)$ must either be 
empty or is a path which stays above the line $x =1$ and hence 
corresponds to the Dyck path $P_2$.  In general, the path $Q_{2j-1}$ 
that starts at $(f_{j-1},0)$ and ends at $(e_j-1,0)$ must stay 
below the $x$-axis so that $Q_{2j-1}$ is the flip of some Dyck path $P_{2j-1}$. 
Similarly, the path $Q_{2j}$ between $(e_j,1)$ and $(f_j-1,1)$ must either be 
empty or is a path which stays above the line $x =1$ and hence 
corresponds to the Dyck path $P_{2j}$. Finally, the path $Q_{2k+1}$ which 
follows $(f_k,0)$ is either empty or is a path that ends at $(2n,0)$ 
and stays below the $x$-axis and, hence, corresponds to the flip of a Dyck path 
$P_{2k+1}$. In this way, we can recover the sequence 
of paths $P_1, \ldots, P_{2k+1}$, which are either empty or Dyck paths, such 
that 
$$P =\overline{P}_1(1,1)P_2(1,-1) 
\overline{P}_3(1,1)P_4(1,-1) \ldots \overline{P}_{2k-1}(1,1)P_{2k}(1,-1)
\overline{P}_{2k+1}.$$
Then for each $i$, $T_i = \phi^{-1}(P_i)$ is either a 
standard tableau or the empty tableau. 
Using these tableaux and being cognizant of the 
restrictions on the initial segments of elements of $\mathcal{F}_{2n,2k}$ 
preceding bricks of size one described above, 
one can easily reconstruct the 2 line 
intermediate arrays corresponding to 
$T_1 e_1 T_2 f_2 \ldots T_{2k-1} e_{2k} T_{2k} f_{2k} T_{2k+1}$.
 For example, 
this process is pictured on line 2 of Figure \ref{fig:unevenbij2}. 
Then we only have to rotate all the bricks of size corresponding to 
a brick of height 2 by 90 degrees to obtain $\beta^{-1}_{n,k}(P)$. 
This step is pictured on line 3 of  \ref{fig:unevenbij2}.

\fig{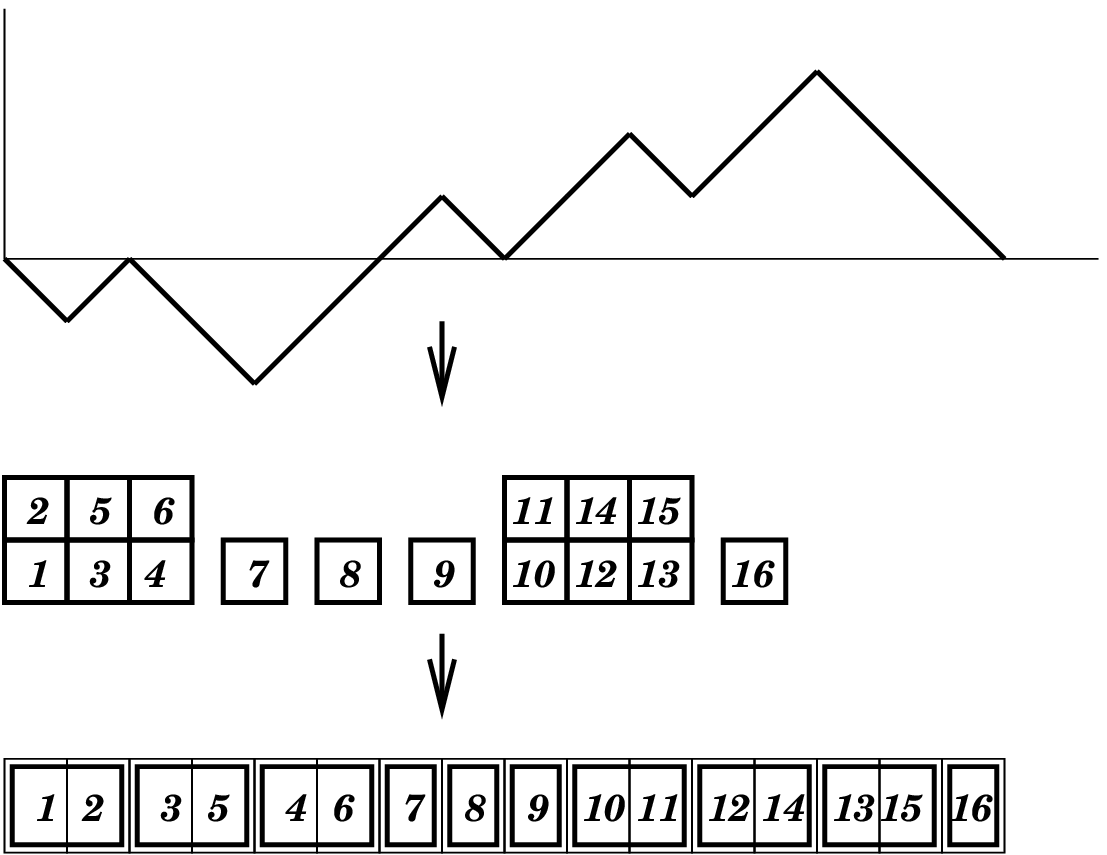}{The bijection $\beta^{-1}_{n,k}$.}

\end{proof}

As a consequence of Theorem \ref{1324,123}, we have the 
closed expression for $\mbox{NM}_{\{1323,123\}}(t,x,y)$. 

\begin{theorem}
$$\mbox{NM}_{\{1323,123\}}(t,x,y) = \left( \frac{1}{U_{\{1323,123\}}(t,y)}  \right)^x  \mbox{~where~} $$
\begin{align*} U_{\{1323,123\}}(t,y) & =1+ \sum_{n \geq 1} 
\frac{t^{2n}}{(2n)!} \left(\sum_{k=0}^n 
\frac{(2k+1)\binom{2n}{n-k}}{n+k+1}(-y)^{n+k}\right)\\ 
& \hspace{60pt} + \sum_{n \geq 0}  \frac{t^{2n+1}}{(2n+1)!} \left(\sum_{k=0}^n \frac{2(k+1)\binom{2n+1}{n-k}}{n+k+2}
(-y)^{n+k+1}\right).\end{align*}

\end{theorem}

\ \\ 
{\bf The proof of Theorem \ref{thm:1324p}.}
\ \\

Let $p \geq 5$ and $\Gamma_p  = \{1324\dots p,123\dots p-1\}$.
It follows from Lemma \ref{lem:keyGamma} that 
any brick in a fixed point of $I_{\Gamma_p}$ has 
size less than or equal to $p-2$. 

Let $(B,\sg)$ be a fixed point of $I_{\Gamma_p}$ where 
$B=(b_1, \ldots, b_t)$ and $\sg = \sg_1 \cdots \sg_n$. 
Suppose that $b_1 = k$ where $1 \leq k \leq p-2$. 
If $b_1=1$, then $\sg_1 =1$ and we can remove 
brick $b_1$ from $(B,\sg)$ and subtract 1 from 
the remaining elements to obtain a fixed point 
$O'$ of $I_{\Gamma_p}$ of length $n-1$.  It is 
easy to see that such fixed points contribute 
$-yU_{\Gamma_p,n-1}(y)$ to $U_{\Gamma_p,}(y)$.

Next assume that $2 \leq k \leq p-2$. 
First we claim that $1, \ldots, k-1$ must be in $b_1$. 
That is, since the minimal elements 
in the bricks increase, reading from 
left to right,  and the elements within each brick 
are increasing, it follows that the first element of 
brick $b_2$ is smaller than every element of $\sg$ to its right. 
Thus if there is an increase between bricks $b_1$ and 
$b_2$, it must be the case the elements in brick $b_1$ are 
the $k$ smallest elements. If there is a decrease 
between bricks $b_1$ and $b_2$, then there must be 
a $1324\dots p$-match that lies in the cells of $b_1$ and 
$b_2$ which must start at position $k-1$.  Thus 
$\sg_{k-1} < \sg_{k+1}$ which means that $\sg_1, \ldots, \sg_{k-1}$ 
must be the smallest $k-1$ elements. We then have 
two cases depending on the position of $k$ in $\sg$.\\ 
\ \\
{\bf Case 1.} $k$ is in the $k^{th}$ cell of $(B,\sg)$.\\
\ \\
In this case, if we remove the entire brick $b_1$ from $(B,\sg)$ and subtract $k$ from the numbers in the remaining cells, we will obtain  a fixed point $O'$ of $I_{\Gamma_p, n-k}.$ It then easily follows that fixed points in Case 1 will contribute $-yU_{\Gamma_p,n-k}(y)$ to $U_{\Gamma_p,n}(y)$.\\
\ \\
{\bf Case 2.} $k$ is in cell $k+1$ of $(B,\sg)$. \\
\ \\
In this case, it is easy to see that $k$ is in the first cell of the second brick in $(B,\sg)$ and there must be a $1324\dots p$-match between the cells of the first two bricks. This match must start from cell $k-1$ in $O$ with the numbers $k-1$ and $k$ playing the roles of $1$ and $2$, respectively, in 
the match. This forces the brick $b_2$ to have exactly $p-2$ cells. Thus  we have two subcases.\\  
\ \\
{\bf Subcase 2.a.} There is no $1324\dots p$-match in $(B,\sg)$ starting at cell $k+p-3$ \\
\ \\
In this case, we claim that $ \{\sg_1, \ldots, \sg_{k+p-2} \}  = \{1,\dots, k+p-2\}$.  That is, we know that the element in the first cell 
of brick $b_3$ is smaller than any of the elements of 
$\sg$ to its right.  Moreover, if there was a decrease between 
brick $b_2$ and $b_3$, then there must be a $1324\dots p$-match 
starting in cell $k+p-3$. Since we are assuming there is not 
such a match this means that there is an increase 
between bricks $b_2$ and $b_3$.  Since the last element of 
$b_2$ must be the largest element in either brick $b_1$ or $b_2$, it follows 
that $ \{\sg_1, \ldots, \sg_{k+p-2} \}  = \{1,\dots, k+p-2\}$.
This forces that 
$\sg_i =i$ for $i \leq k-1$, $\sg_{k}=k+1$, $\sg_{k+1} =k$, 
$\sg_{k+2} = k+2$, $\sg_i = i$ for $k+2 < i \leq k+p-2$. 
Hence, the first two bricks of $(B,\sg)$ are completely determined. It then follows that if we let $O'$ be the result by removing the first $k+p-2$ cells from 
$(B,\sg)$ and subtracting  $k+p-2$ from the numbers in the remaining cells, then $O'$ will be a fixed point in $\mathcal{O}_{\Gamma_p,n-k-(p-2)}$.  
It then easily follows that  fixed points in Subcase 2.a 
contribute $(-y)^2U_{\Gamma_p,n-k-(p-2)}(y)$ to $U_{\Gamma_p,n}(y)$.\\ 
\ \\
{\bf Subcase 2.b.} There is a $1324\dots p$-match in $(B,\sg)$ starting at cell $k+p-3$ \\
\ \\
In this case, it must be that $\sg_{k+p-3}< \sg_{k+p-1} < \sg_{k+p-2}$ so 
that there is a decrease between bricks $b_2$ and $b_3$. This 
means that the $1324\dots p$-match starting in cell 
$k+p-3$ must be contained in bricks $b_2$ and $b_3$. In particular, 
this means that $b_3 =p-2$. In general, suppose that the bricks $b_2, \dots, b_{m-1}$ all have exactly $p-2$ cells and let $c_i = k+(i-1)(p-2)-1$ for all $1 \leq i \leq m-1,$ so that $c_i$ is the second-to-last cell of brick $b_i.$ In addition, suppose there are $1324\dots p$-matches starting at cells $c_1,c_2,\dots, c_{m-1}$ but there are no $1324\dots p$-match starting at cell $c_{m} = k-(m-1)(p-2)-1$ in $O.$ We then have the situation pictured in Figure \ref{fig:132p} below.


\fig{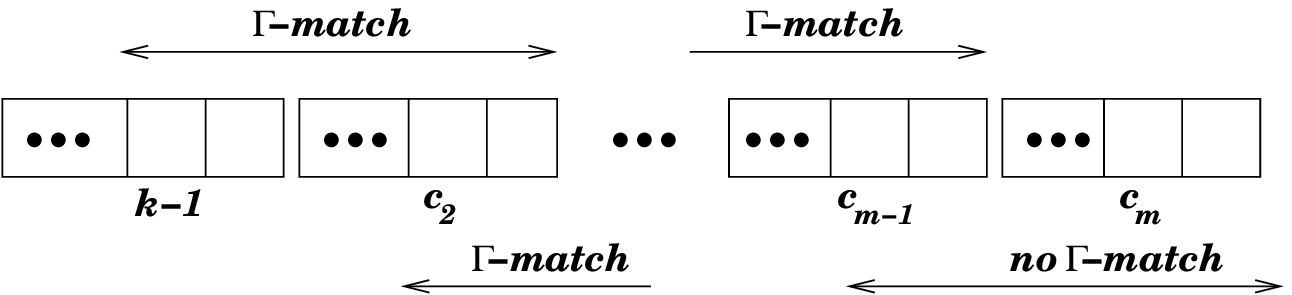}{A fixed point with $\Gamma_p$-matches starting at $c_{i}$ for $i = 1,\dots, m-1$.}

First, we claim that 
$\{\sg_1,\sg_2,\dots, \sg_{c_{m+1}}\} = \{1,2,\dots, c_{m+1}\}$. 
Since there is no $\Gamma_p$-match starting at $\sg_{c_m}$ in $\sg$, it 
cannot be that there is decrease between brick $b_{m}$ and 
$b_{m+1}$.  Because the minimal elements in the bricks of $B$ 
increase, reading from left to right, and the elements in each 
brick increase, it follows that $\sg_{c_{m}+2}$, which 
is the first element of brick $b_{m+1}$, is smaller than all 
the elements to its right. On the other hand, 
because there are $1324 \cdots p$-matches starting in $\sg$ starting 
at $c_1, \ldots, c_{m-1}$ it follows that $\sg_{c_{m}+1}$, 
which is last cell in brick $b_{m}$, is greater than all 
elements of $\sg$ to its left.   It follows that 
$\{\sg_1,\sg_2,\dots, \sg_{c_{m+1}}\} = \{1,2,\dots, c_{m+1}\}$.

Next we claim that we can prove by induction that  $\sg_{c_i} = c_i$ 
and $\{\sg_1, \ldots, \sg_{c_i}\} = \{1, \ldots, c_i\}$ for $1 \leq i \leq m$.
Our arguments above show that $\sg_i =i$ for $i=1, \ldots, k-1=c_1$. 
Thus the base case holds.  So assume that $\sg_{c_{j-1}} = c_{j-1}$, for $1 \leq i \leq j$, and $\{ \sg_1, \sg_2, \ldots, s_{c_{j-1}}   \} = \{1,2,\dots, c_{j-1}\}$.  Since there is a $132\cdots p$-match in $\sg$ 
starting at position $c_{j-1}$ and $p \geq 5$, 
it must be the case that all the numbers 
$\sg_{c_{j-1}}, \sg_{c_{j-1}+1}, \ldots, \sg_{c_{j-1}+p-3}$ are all less than 
$\sg_{c_j} = \sg_{c_{j-1}+p-2}.$ Since $\{\sg_1,\sg_2, \ldots, \sg_{c_{j-1}}   \} = \{1,2,\dots, c_{j-1}\}$, we must have $\sg_{c_j} \geq c_j.$ If $\sg_{c_j}>c_j$, then  let $d$ be the smallest number from $\{ 1,2,\dots, c_j \}$ that does not belong to the bricks $b_1, \dots, b_j.$ Since the numbers in a brick increase and the first cells of the bricks form an increasing sequence, it must be the case that $d$ is in the first cell of brick $b_{j+1},$ namely $\sg_{c_j+2} = d.$ We have two possibilities for $j$. 
\begin{enumerate}
\item If $j < m,$ then $\sg_{c_j +2} = d < c_j \leq \sg_{c_j}$. This contradicts the assumption that there is a $1324\dots p$-match starting from cell $c_j$ in $\sg$ for $\sg_{c_j}$ needs to play the role of 1 in such a match.

\item If $j = m,$ then there is a descent between the bricks $b_m$ and $b_{m+1}$ and there must be a $1324\dots p$-match that lies entirely in the cells of $b_m$ and $b_{m+1}$ in $O.$ However, the only possible match must start from cell $c_m,$ the second-to-last cell in $b_m.$ This contradicts our assumption that there is no match starting from cell $c_m$ in $O.$
\end{enumerate} 
Hence, 
$\sg_{c_j} = c_j$ and $\{\sg_1, \sg_2, \ldots, \sg_{c_j} \} = \{1,2,\dots, c_j\}.$ for $1 \leq j \leq m$. 

We claim that the values of $\sg_i$ are forced for 
$i \leq c_m+1$.  That is, consider the 
first $1324\cdots p$-match starting at position 
$k-1$.  Since $p \geq 5$, we know that $\sg_{k+p-2} =k+p-2 > \sg_{k+2}$. 
This forces that $\sg_{k} = k+1$, $\sg_{k+1} =k$, $\sg_{k+2} =k+2$ 
so that the values of $\sg_i$ for $i \leq k+p-2$.  This type of 
argument can be repeated for all the remaining $1324\cdots p$-matches 
starting at $c_2, \ldots, c_{m-1}$. 
Thus if we remove the first  $k+(m-1)(p-2)$ cells of $O$, we obtain a fixed point $O'$ of $I_{\Gamma_p}$ in $\mathcal{O}_{\Gamma_p, n - k - (m-1)(p-2)}$. 
On the other 
hand, suppose that we start with a fixed point $(D,\tau)$ of $I_{\Gamma_p}$ in 
 $\mathcal{O}_{\Gamma_p, n - k - (m-1)(p-2)}$ where 
$D=(d_1, \ldots,d_r)$ and $\tau = \tau_1, \ldots, \tau_{n - k - (m-1)(p-2)}$.
Let $\overline{\tau} = \overline{\tau}_1 \cdots \overline{\tau}_{n - k - (m-1)(p-2)}$ be the result of adding $n - k - (m-1)(p-2)$ to every element of 
$\tau$. Then it is easy to see that $(B,\sg)$ is a fixed point 
of $I_{\Gamma_p}$, where 
$B =(k,(p-2)^m,d_1, \ldots,d_r)$ and 
$\sg = \sg_1 \cdots \sg_{k+(m-1)p-2}\overline{\tau}$ where 
$\sg_1 \cdots \sg_{k+(m-1)(p-2)}$ is the unique permutation 
in $\mathfrak{S}_{k+(m-1)(p-2)}$ with $1324\cdots p$-matches starting 
at positions $c_1, \ldots, c_{m-1}$. It follows that the contribution of 
the fixed points in Case 2.b to $U_{\Gamma_p,n}(y)$ is $\sum_{m \geq 3}(-y)^mU_{\Gamma_p, n-k-(m-1)(p-2)(y)}$.

Hence, for any fixed point $O_k$ that has $k$ cells in the first brick, for $1\leq k \leq p-2,$ the contribution of $O_k$ to $U_{\Gamma_p,n}(y)$ is 
\[ \displaystyle (-y)U_{\Gamma_p,n-k}(y) + \sum_{m=2}^{\lfloor\frac{n-k}{p-2}\rfloor }(-y)^{m}U_{\Gamma_p, n-k-(m-1)(p-2)}(y).   \]

Therefore, we obtain the following recursion for $U_{\Gamma_p,n}(y)$ 
\begin{align*}
\displaystyle  U_{\Gamma_p,n}(y) &= \sum_{k=1}^{p-2}(-y)U_{\Gamma_p,n-k}(y) + \sum_{k=1}^{p-2}\sum_{m=2}^{\lfloor\frac{n-k}{p-2}\rfloor }(-y)^{m}U_{\Gamma_p, n-k-(m-1)(p-2)}(y). 
\end{align*} 
This completes the proof of Theorem \ref{thm:1324p}. \qed

\section{Conclusion and Problems for Future Research} 

In this paper, we have shown that the reciprocal method introduced by Jones and Remmel in \cite{JR1} can be extended to a family $\Gamma$ whose permutations all start with 1 and have at most one descent. Specifically, we have proved if $$\Gamma = \Gamma_{k_1,k_2} =  \{\sigma \in \mathfrak{S}_p: \sigma_1=1, \sigma_{k_1+1}=2, \sigma_1 < \sigma_2< \cdots<\sigma_{k_1}~ \&~\sigma_{k_1+1} < \sigma_{k_1+2}< \cdots<\sigma_{p} \}$$ where $k_1,k_2 \geq 2$, 
$\Gamma = \Gamma_{k_1,k_1,s} = \Gamma_{k_1,k_2} \cup \{1 \cdots s(s+1)\}$ 
where $s \geq k_1\geq 2$, or   $\Gamma = \Gamma_{p} = \{1324\cdots p , 123 \cdots p-1\}$ where $p \geq 4$, then the polynomials $U_{\Gamma, n}(y)$ satisfy simple recursions and these recursions can be used to compute the terms in the generating function 
\begin{equation*} 
\mbox{NM}_{\Gamma}(t,x,y)= \sum_{n \geq 0} \frac{t^n}{n!} \sum_{\sg \in \mathcal{NM}_n(\Gamma)}x^{\LRmin{\sg}}y^{1+\des(\sg)}.
\end{equation*} 

From the values of the polynomials $U_{\Gamma,n}(y)$ computed through Mathematica, we conjecture that the polynomials $U_{\Gamma,n}(y)$ are log-concave for $\Gamma = \{1324,1423\}$ and $\Gamma = \{1324\cdots p, 123\cdots p\}$, where $p \geq 4$. 
However, the polynomials $U_{\Gamma_{k_1,k_2},n}(-y)$ are not always log-concave when 
$k_1$ is larger than $k_2$.

The next set of problems to consider is to show that the same machinery can be extended to families $\Gamma$ of  permutations which all start with 1 
but may have more than one descent. This type of problem 
in the case where $\Gamma$ consists of single permutation 
$\tau$ was first mentioned by Jones and Remmel in \cite{JR3}, where the authors gave a recursion for the polynomial $U_{\tau,n}(y)$ for $\tau = 15243.$

The main problem when the permutations in a family  $\Gamma$ are allowed to 
have more than one descent is that the mapping $I_{\Gamma}$ defined 
in Section 3 is no longer an involution. To see this, suppose the permutations in $\Gamma$ have more than one descents and consider the case where we have a decrease between the last cell of brick $b_{i-1}$ and the first cell of brick $b_i$, but we are unable to combine them since there is a $\Gamma$-match that involves the cells of bricks $b_{i-1}$ and $b_i.$ In this case, brick $b_i$ will have at least one cell labeled with $y$. According to the current mapping, we will try to split  
brick $b_i$ after some cell $c$ labeled with $y$ into two bricks: $b'$, containing all the cells of 
$b_i$ up to and including $c$, and $b''$, containing all the remaining 
cells of $b_i$. Then, we will be able to combine 
$b'$ with $b_{i-1}$ because there is still a decrease 
between $b_{i-1}$ and $b'$ but now there is no $\Gamma$-match 
that lies in the cells of $b_{i-1}$ and $b'$. This means 
that we cannot use cell $c$ in a definition of an involution. 
Thus we must restrict ourselves to cells $c$ labeled with $y$ 
which do not have this property. The result of this 
restriction is that the fixed points are more complicated 
than before.  In particular, we can no longer guarantee 
that if $(B,\sg)$ is a fixed point of such an involution,  
then $\sg$ is increasing in the bricks of $B$. Nevertheless 
one can analyze the fixed points of such an involution 
for certain simple permutations $\tau$ and simple 
families of permutations $\Gamma$. 
For example, we can prove the following results. 

\begin{theorem}
For $\tau = 1432,$ $U_{\tau,1}(y)=-y$, and for $n \geq 2,$
\begin{align*}
\displaystyle U_{\tau,n}(y)  = & ~  (1-y)U_{\tau,n-1}(y) - y^2 \binom{n-2}{2}U_{\tau, n-3}(y).  
\end{align*}
\end{theorem} 

\begin{theorem}
For $\tau = 142536,$ $U_{\tau,1}(y)=-y$, and for $n \geq 2,$
\begin{align*}
\displaystyle U_{\tau,n}(y) = &~~~  (1-y)U_{\Gamma,n-1}(y) +  \sum_{k=1}^{\Floor[(n-2)/6]} H_{2k}y^{3k}U_{n-6k-1}(y) \\
& ~~~\qquad  - \sum_{k=1}^{\Floor[n/6]} H_{2k-1}y^{3k-1}\left[U_{\tau,n-6k+2}(y) +yU_{\tau,n-6k+1}(y) \right]
\end{align*}
where $H_i$ is the determinant the matrix of Catalan numbers, given by the following formulas.
\begin{equation*} H_{2k-1} =\begin{vmatrix}
C_2 & C_5 & C_8 & C_{11} & \cdots & C_{3k-4} & C_{3k-2} \\ 
-1 & C_2 & C_4 & C_8 & \cdots & C_{3k-7} & C_{3k-5} \\ 
0 & -1 & C_2 & C_5 & \cdots & C_{3k-10} & C_{3k-8} \\ 
0 & 0 & -1 & C_2 & \cdots & C_{3k-13} & C_{3k-11} \\ 
\vdots & \vdots & \vdots & \vdots &  & \vdots & \vdots \\ 
0 & 0 & 0 & 0 & \cdots & C_2 & C_4 \\ 
0 & 0 & 0 & 0 & \cdots & -1 & 1
\end{vmatrix} , \mbox{~and~} 
\end{equation*}

\begin{equation*} H_{2k} = \begin{vmatrix}
C_2 & C_5 & C_8 & C_{11} & \cdots & C_{3k-4} & C_{3k-1} \\ 
-1 & C_2 & C_5 & C_8 & \cdots & C_{3k-7} & C_{3k-4} \\ 
0 & -1 & C_2 & C_5 & \cdots & C_{3k-10} & C_{3k-7} \\ 
0 & 0 & -1 & C_2 & \cdots & C_{3k-13} & C_{3k-10} \\ 
\vdots & \vdots & \vdots & \vdots &  & \vdots & \vdots \\ 
0 & 0 & 0 & 0 & \cdots & C_2 & C_5 \\ 
0 & 0 & 0 & 0 & \cdots & -1 & C_2
\end{vmatrix}.
\end{equation*}
\end{theorem}

\begin{theorem}
For $\tau = 162534,$ $U_{\tau,1}(y)=-y$, and for $n \geq 2,$
\begin{align*}
\displaystyle U_{\tau,n}(y)  = & ~  (1-y)U_{\tau,n-1}(y)  - \sum_{k = 1}^{\Floor[n/6]}y^{3k-1}\binom{n-3k-1}{3k-1}U_{\tau, n-6k+1}(y) \\
& \qquad \quad +\sum_{k = 1}^{\Floor[(n-2)/6]}y^{3k}\binom{n-3k-2}{3k}U_{\tau, n-6k-1}(y).  
\end{align*}
\end{theorem} 

\begin{theorem}
For $\Gamma = \{14253,15243\},$ $U_{\Gamma,1}(y)=-y$, and for $n \geq 2,$
\begin{align*}
\displaystyle  U_{\Gamma,n}(y) & = (1-y)U_{\Gamma,n-1}(y)  -y^2(n-3)\left(U_{\Gamma, n-4}(y) +(1-y)(n-5)U_{\Gamma, n-5}(y)  \right)\\
& \qquad \qquad \qquad \qquad \qquad -y^3(n-3)(n-5)(n-6)U_{\Gamma,n-6}(y). 
\end{align*}
\end{theorem} 
These results will appear in subsequent papers.

The authors would like to thank the anonymous referees whose comments helped to improve the presentation of this paper.


\end{document}